\documentclass[11pt]{article}
\usepackage[utf8]{inputenc}
\usepackage{amsmath}
\usepackage[hmargin=1in,vmargin=1in]{geometry}
\usepackage{graphics}
\usepackage[toc,page]{appendix}
\usepackage{amsfonts}
\usepackage{MnSymbol}
\usepackage{wasysym}
\usepackage{enumerate}
\usepackage{amsthm}
\usepackage{graphicx}
\usepackage{float}
\usepackage{chngcntr}
\theoremstyle{plain}
\newtheorem{thm}{Theorem}[section]
\newtheorem{lemma}[thm]{Lemma}
\newtheorem{cor}[thm]{Corollary}
\newtheorem{prop}[thm]{Proposition}
\newtheorem{conj}[thm]{Conjecture}
\theoremstyle{definition}
\newtheorem{defn}[thm]{Definition}

\newtheorem{remark}[thm]{Remark}

\numberwithin{equation}{section}
\counterwithin{figure}{section}
\title{Conjectures for the Integral Moments and Ratios of L-functions in Even characteristic}
\author{J. MacMillan}
\date{}
\newcommand\blfootnote[1]{%
  \begingroup
  \renewcommand\thefootnote{}\footnote{#1}%
  \addtocounter{footnote}{-1}%
  \endgroup
  }
\begin{document}
\maketitle
\blfootnote{Date: September 30, 2021}
\par\noindent
ABSTRACT: In this paper, we extend to the function field setting the heuristics developed by Conrey, Farmer, Keating, Rubinstein and Snaith for the integral moments of L-functions. Also, we adapt to function field setting  the heuristics first developed by Conrey, Farmer and Zirnbauer to the study of mean values of ratios of L-functions. Specifically, we obtain an asymptotic formula for the integral moments and ratios of the quadratic Dirichlet L-functions $L(s,\chi_u)$ over the rational function field $\mathbb{F}_q(T)$, when $q$ is a power of 2 and over a given family. As an application, we calculate the one-level density for the zeros of these L-functions. 
\section{Introduction}
An important and well-studied problem in Analytic Number Theory is to understand the asymptotic behaviour of moments of families of L-functions. In the case of the Riemann-zeta function, the problem is to understand the asymptotic behaviour of 
\begin{equation}
    M_k(T)=\int_0^T\left|\zeta\left(\frac{1}{2}+it\right)\right|^{2k}dt,
\end{equation}
as $T\rightarrow\infty$. In this context, Hardy and Littlewood \cite{Hardy1916} established that
\begin{equation*}
    M_1(T)\sim T\log T
\end{equation*}
and Ingham \cite{Ingham1926} established
\begin{equation*}
    M_2(T)\sim\frac{1}{2\pi^2}T\log^4T.
\end{equation*}
For $k\geq 3$, the problem is still open. However it is believed that
\begin{equation*}
    M_k(T)\sim c_kT(\log T)^{k^2}
\end{equation*}
for a positive constant $c_k$. Conrey and Ghosh \cite{ConreyGhosh1992} presented the constant $c_k$ in a more explicit form and Keating and Snaith \cite{Keating2000a} conjectured a precise formula for $c_k$ for $\Re(s)>-\frac{1}{2}$ based on the analogy with the characteristic polynomials of random matrices. Conrey, Farmer, Keating, Rubinstein and Snaith \cite{Conrey2008} described an algorithm for obtaining explicit expressions for lower terms for the conjectured full asymptotics of the moments of the Riemann zeta function.   \\
\par\noindent
In a recent series of papers, Conrey and Keating \cite{ConreyKeatingMomentsofzetaI,ConreyKeatingMomentsofzetaII,ConreyKeatingMomentsofzetaIII,ConreyKeatingMomentsofzetaIV,ConreyKeatingMomentsofzetaV} examined the calculations of the $2k$th moment and shifted moments of the Riemann zeta function on the critical line from a number theoretic perspective. Their method was to approximate $\zeta(s)^k$ by a long Dirichlet polynomial and then to compute the mean square of the Dirichlet polynomial. In particular, the series of papers takes care of the off diagonal terms and explain their role towards the contribution of the main terms in the asymptotic formula. \\
\par\noindent
For the family of Dirichlet L-functions $L(s,\chi_d)$ with $\chi_d$ a real primitive Dirichlet character modulo $d$ defined by the Jacobian symbol $\chi_d(n)=\left(\frac{d}{n}\right)$, the problem is to understand the asymptotic behaviour of 
\begin{equation}
    \sum_{0<d\leq D} L\left(\frac{1}{2},\chi_d\right)^k,
\end{equation}
summing over fundamental discriminants $d$, as $D\rightarrow\infty$. In this context, Jutila \cite{Jutila1981} proved that
\begin{equation}
    \sum_{0<d\leq D}L\left(\frac{1}{2},\chi_d\right)=\frac{P(1)}{4\zeta(2)}D\left[\log\left(\frac{D}{\pi}\right)+\frac{\Gamma'}{\Gamma}\left(\frac{1}{4}\right)+4\gamma-1+4\frac{P'}{P}(1)\right]+O(D^{\frac{3}{4}+\epsilon}),
\end{equation}
and 
\begin{equation}
    \sum_{0<d\leq D}L\left(\frac{1}{2},\chi_d\right)^2=\frac{c}{\zeta(2)}D\log^3D+O(D(\log D)^{\frac{5}{2}+\epsilon}),
\end{equation}
for all $\epsilon>0$ with
\begin{equation*}
    P(s)=\prod_p\left(1-\frac{1}{p^s(p+1)}\right)
\end{equation*}
and
\begin{equation*}
    c=\frac{1}{48}\prod_p\left(1-\frac{4p^2-3p+1}{p^4+p^3}\right).
\end{equation*}
When restricting $d$ to be odd, square-free and positive, so that $\chi_{8d}$ are real, primitive characters with conductor $8d$ and with $\chi_{8d}(-1)=1$, Soundarajan \cite{Soundararajan2000NonvanishingS=1/2} proved that there exists polynomials $P_2$ and $P_3$ of degree 3 and 6 respectively such that
\begin{equation*}
    \sum_{0<d\leq D}L\left(\frac{1}{2},\chi_{8d}\right)^2=DP_2(\log D)+O(D^{\frac{5}{6}+\epsilon})
\end{equation*}
and 
\begin{equation}\label{eq:1.5}
    \sum_{0<d\leq D}L\left(\frac{1}{2},\chi_{8d}\right)^3=DP_3(\log D)+O(D^{\frac{11}{12}+\epsilon}),
\end{equation}
where the sum is over fundamental discriminants $8d$. Using multiple Dirichlet series Diaconu, Goldfeld and Hoffstein  \cite{Diaconu2004} improved the error term of (\ref{eq:1.5}) to $D^{\theta+\epsilon}$ for some explicit constant $\theta$. Furthermore Diaconu and Whitehead \cite{Diaconu2018} proved, for the third moment, that there exists a secondary main term of size $D^{\frac{3}{4}}$, whilst showing that the error term is of the order $O(D^{\frac{2}{3}+\delta})$ for every $\delta>0$. For the fourth moment Shen \cite{Shen2021} proved, under the assumption of the Generalised Riemann hypothesis for $L(s,\chi_d)$ for all fundamental discriminants $d$ and the Riemann hypothesis for $\zeta(s)$, that
\begin{equation}
    \sum_{\substack{0<d\leq D\\(d,2)=1}}L\left(\frac{1}{2},\chi_{8d}\right)^4=CD(\log D)^{10}+O(D(\log D)^{\frac{39}{4}+\epsilon}),
\end{equation}
for some positive constant $C$.\\
\par\noindent
For the sum over fundamental discriminants $d$, it is conjectured that
\begin{equation*}
    \sum_{0<d\leq D}L\left(\frac{1}{2},\chi_{d}\right)^k\sim C_kD(\log D)^{\frac{k(k+1)}{2}}.
\end{equation*}
 Extending their approach to the zeta-function moments, Keating and Snaith \cite{Keating2000} conjectured a precise value for $C_k$.\\
 \par \noindent
 Conrey, Farmer, Keating, Rubinstein and Snaith \cite{Conrey2005} presented a new heuristic to conjecture all principal lower order terms. For the family of quadratic Dirichlet L-functions, they conjectured the following.
 \begin{conj}
 Let $X_d(s)=|d|^{\frac{1}{2}-s}X(s,a)$, where $a=0$ if $d>0$ and $a=1$ if $d<0$ and
 \begin{equation*}
     X(s,a)=\pi^{s-\frac{1}{2}}\frac{\Gamma\left(\frac{1+a-s}{2}\right)}{\Gamma\left(\frac{s+a}{2}\right)}.
 \end{equation*}
 That is $X_d(s)$ is the factor in the functional equation for the quadratic Dirichlet L-function
 \begin{equation*}
     L(s,\chi_d)=\epsilon_dX_d(s)L(1-s,\chi_d).
 \end{equation*}
 Summing over fundamental discriminant $d$
 \begin{equation}
     \sum_{d}L\left(\frac{1}{2},\chi_d\right)^k=\sum_dQ_k(\log|d|)(1+o(1)),
 \end{equation}
 where $Q_k$ is the polynomial of degree $\frac{1}{2}k(k+1)$ given by the $k$-fold residue
 \begin{equation}
     Q_k(x)=\frac{(-1)^{\frac{k(k-1)}{2}}2^k}{k!}\frac{1}{(2\pi i)^k}\oint\dotsc\oint\frac{G(z_1,\dotsc,z_k)\Delta(z_1^2,\dotsc,z_k^2)^2}{\prod_{j=1}^kz_j^{2k-1}}e^{\frac{x}{2}\sum_{j=1}^kz_j}dz_1\dotsc dz_k,
 \end{equation}
 with
 \begin{equation*}
     G(z_1,\dotsc,z_k)=A_k(z_1,\dotsc,z_k)\prod_{j=1}^kX\left(\frac{1}{2}+z_j,a\right)^{-\frac{1}{2}}\prod_{1\leq i\leq j\leq k}\zeta(1+z_i+z_j),
 \end{equation*}
 $\Delta(z_1,\dotsc,z_k)$ is the Vandermonde determinant given by
 \begin{equation}\label{eq:1.9}
     \Delta(z_1,\dotsc,z_k)=\prod_{1\leq i<j\leq k}(z_j-z_i),
 \end{equation}
 and $A_k$ is the Euler product, absolutely convergent for $|\Re(z_j)|<\frac{1}{2}$ defined by
 \begin{align*}
     A_k(z_1,\dotsc,z_k)&=\prod_p\prod_{1\leq i\leq j\leq k}\left(1-\frac{1}{p^{1+z_i+z_j}}\right)\\
     &\times \left(\frac{1}{2}\left(\prod_{j=1}^k\left(1-\frac{1}{p^{\frac{1}{2}+z_j}}\right)^{-1}+\prod_{j=1}^k\left(1+\frac{1}{p^{\frac{1}{2}+z_j}}\right)^{-1}\right)+\frac{1}{p}\right)\left(1+\frac{1}{p}\right)^{-1}.
 \end{align*}
 \end{conj}
 \par\noindent
 Conrey, Farmer and Zirnbauer \cite{Conrey2008a} presented a generalisation of the heuristic arguments used in \cite{Conrey2005} leading to conjectures for the ratios of products of L-functions. These conjectures are very useful since it is possible to obtain from them all $n$-level correlations of zeros with lower order terms. For the family of Dirichlet L-functions they produced the following conjecture.
 \begin{conj}
 Let $\mathcal{D}^+=\{L(s,\chi_d):d>0\}$ to be symplectic family of L-functions associated with the quadratic character $\chi_d$, and suppose that the real parts of $\alpha_k$ and $\gamma_q$ are positive. Then
 \begin{align*}
     &\sum_{0<d\leq D}\frac{\prod_{k=1}^KL\left(\frac{1}{2}+\alpha_k,\chi_d\right)}{\prod_{q=1}^QL\left(\frac{1}{2}+\gamma_q,\chi_d\right)},\\
     &=\sum_{0<d\leq D}\sum_{\epsilon\in\{-1,1\}^K}\left(\frac{|d|}{\pi}\right)^{\frac{1}{2}\sum_{k=1}^K(\epsilon_k\alpha_k-\alpha_k)}\\
     &\times\prod_{k=1}^Kg_+\left(\frac{1}{2}+\frac{\alpha_k-\epsilon_k\alpha_k}{2}\right)Y_{\mathfrak{D}^+}(\epsilon_1\alpha_1,\dotsc,\epsilon_K\alpha_K;\gamma)A_{\mathfrak{D^+}}(\epsilon_1\alpha_1,\dotsc,\epsilon_K\alpha_K;\gamma)+o(D),
\end{align*}
where 
\begin{equation*}
    g_+(s)=\frac{\Gamma\left(\frac{1-s}{2}\right)}{\Gamma\left(\frac{s}{2}\right)},
\end{equation*}
\begin{equation*}
    Y_{\mathfrak{D^+}}(\alpha;\gamma)=\frac{\prod_{1\leq j\leq k\leq K}\zeta(1+\alpha_j+\alpha_k)\prod_{1\leq q<r\leq Q}\zeta(1+\gamma_q+\gamma_r)}{\prod_{k=1}^K\prod_{q=1}^Q\zeta(1+\alpha_k+\gamma_q)}
\end{equation*}
and
\begin{align*}
    A_{\mathfrak{D}^+}(\alpha;\gamma)&=\prod_p\frac{\prod_{1\leq j\leq k\leq K}\left(1-\frac{1}{p^{1+\alpha_j+\alpha_k}}\right)\prod_{1\leq q<r\leq Q}\left(1-\frac{1}{p^{1+\gamma_q+\gamma_r}}\right)}{\prod_{k=1}^K\prod_{q=1}^Q\left(1-\frac{1}{p^{1+\alpha_k+\gamma_q}}\right)}\\
    &\times\left(1+\left(1+\frac{1}{p}\right)^{-1}\sum_{0<\sum_ka_k+\sum_qc_q\text{ is even}}\frac{\prod_{q=1}^Q\mu(p^{c_q})}{p^{\sum_ka_k\left(\frac{1}{2}+\alpha_k\right)+\sum_qc_q\left(\frac{1}{2}+\gamma_q\right)}}\right).
\end{align*}
\end{conj}
\par\noindent
Using the Ratios conjecture 1.2 Conrey and Snaith \cite{Conrey2007} computed the one-level density for zeros of quadratic Dirichlet L-functions complete with lower order terms. Namely, they considered the one-level density
\begin{equation*}
    S_1(f):=\sum_{d\leq D}\sum_{\gamma_d}f(\gamma_d),
\end{equation*}
where $f(z)$ is holomorphic throughout the strip $|\Im(z)|<2$ is real on the real line, $f(x)\ll\frac{1}{(1+x^2)}$ as $x\rightarrow\infty$ and $\gamma_d$ denotes the ordinate of a generic zero of $L(s,\chi_d)$ on the half line. For the family of Dirichlet L-functions they obtained the following result.
\begin{thm}
Assuming the Ratios conjecture 1.2 and $f$ satisfying the conditions given above, we have,
\begin{align*}
    S_1(f):&=\sum_{d\leq D}\sum_{\gamma_d}f(\gamma_d)=\frac{1}{2\pi}\int_{-\infty}^{\infty}f(t)\sum_{d\leq D}\Bigg(\log\frac{d}{\pi}+\frac{1}{2}\frac{\Gamma'}{\Gamma}\left(\frac{1}{4}+\frac{it}{2}\right)+\frac{1}{2}\frac{\Gamma}{\Gamma}\left(\frac{1}{4}-\frac{it}{2}\right)\\
    &+2\Bigg(\frac{\zeta'(1+2it)}{\zeta(1+2it)}+A_{\mathfrak{D}^+}'(it;it)-\left(\frac{d}{\pi}\right)^{-it}\frac{\Gamma\left(\frac{1}{4}-\frac{it}{2}\right)}{\Gamma\left(\frac{1}{4}+\frac{it}{2}\right)}\zeta(1-2it)A_{\mathfrak{D}^+}(-it,it)\Bigg)\Bigg)dt+o(D),
\end{align*}
where
\begin{equation*}
    A_{\mathfrak{D}^+}(-r;r)=\prod_p\left(1-\frac{1}{(p+1)p^{1-2r}}-\frac{1}{p+1}\right)\left(1-\frac{1}{p}\right)^{-1}
\end{equation*}
and
\begin{equation*}
    A'_{\mathfrak{D}^+}(r;r)=\sum_p\frac{\log p}{(p+1)(p^{1+2r}-1)}.
\end{equation*}
\end{thm}
\section{Dirichlet L-Functions in Function Fields}
\subsection{Background in $\mathbb{F}_q[T]$}
We first introduce the notion which will be used throughout this article and then provide some background information on Dirichlet L-functions in function fields in both the odd and even characteristic cases. Most of the facts in this subsection are proved in \cite{Rosen2002}. Let $k:=\mathbb{F}_q(T)$ be the rational function field and $\infty=\left(\frac{1}{T}\right)$ be the infinite prime of $k$. We denote $\mathbb{A}^+$ to be the set of all monic polynomials in $\mathbb{A}:=\mathbb{F}_q[T]$ and $\mathbb{A}^+_n$ and $\mathbb{A}^+_{\leq n}$ the sets of all monic polynomials of degree $n$ and degree at most $n$ in $\mathbb{F}_q[T]$ respectively. For $f\in\mathbb{F}_q[T]$, we denote its norm by $|f|=q^{\text{deg}(f)}$ and $sgn(f)$ be its leading coefficient. Similarly, we let $\mu(f)$ and $\phi(f)$ denote the M\"obius function and the Euler-$\phi$ funcion in $\mathbb{F}_q[T]$ respectively. The letter $P$ denotes a monic irreducible polynomial over $\mathbb{F}_q[T]$ and we let $\mathcal{P}$ denote the set of all monic irreducible polynomials in $\mathbb{F}_q[T]$.\\
\par\noindent
For $\Re(s)>1$, the zeta function of $\mathbb{A}$, denoted by $\zeta_{\mathbb{A}}(s)$ is defined by the infinite series
\begin{equation}
    \zeta_{\mathbb{A}}(s):=\sum_{f\in\mathbb{A}^+}\frac{1}{|f|^s}=\prod_P\left(1-|P|^{-s}\right)^{-1}.
\end{equation}
There are $q^n$ monic polynomials of degree $n$, therefore 
\begin{equation*}
\zeta_{\mathbb{A}}(s)=(1-q^{1-s})^{-1}.
\end{equation*}
\subsection{Quadratic Function Field in Odd Characteristic}
For this subsection, let $q$ be odd and $q\equiv 1(\text{mod }4)$. In this setting, when $D\in\mathbb{F}_q[T]$ is square-free, the quadratic Dirichlet character $\chi_D$ is defined by the quadratic residue symbol for $\mathbb{F}_q[T]$ and the L-function corresponding to $\chi_D$ is defined for $\Re(s)>1$ by the infinite series
\begin{equation}
    L(s,\chi_D):=\sum_{f\in\mathbb{A}^+}\frac{\chi_D(f)}{|f|^s}.
\end{equation}
Let $\mathcal{H}_n$ denote the set of all monic, square-free polynomials of degree $n$ in $\mathbb{A}$. In this case Andrade and Keating \cite{Andrade2012} computed an asymptotic formula for the first moment of the family of L-functions associated to the quadratic character $\chi_D$, with $D\in\mathcal{H}_{2g+1}$. In particular, they proved that
\begin{equation}\label{eq:2.3}
    \sum_{D\in\mathcal{H}_{2g+1}}L\left(\frac{1}{2},\chi_D\right)=\frac{P(1)}{2\zeta_{\mathbb{A}}(2)}|D|\left[\log_q|D|+1+\frac{4}{\log q}\frac{P'}{P}(1)\right]+O\left(|D|^{\frac{3}{4}+\frac{1}{2}\log_q2}\right),
\end{equation}
where
\begin{equation*}
    P(s)=\prod_P\left(1-\frac{1}{|P|^s(|P|+1)}\right).
\end{equation*}
Andrade and Keating \cite{Andrade2014} also adapted the recipe of \cite{Conrey2005,Conrey2008a} to the function field setting to conjecture the integral moments and ratios of quadratic Dirichlet L-functions in function fields. Their conjecture reads.
\begin{conj}
Suppose that $q$ odd is the fixed cardinality of the finite field $\mathbb{F}_q$ and let $\mathcal{X}_D(s)=|D|^{\frac{1}{2}-s}X(s)$ and 
\begin{equation*}
    X(s)=q^{-\frac{1}{2}+s}.
\end{equation*}
That is $\mathcal{X}_D(s)$ is the factor in the functional equation 
\begin{equation*}
    L(s,\chi_D)=\mathcal{X}_D(s)L(1-s,\chi_D).
\end{equation*}
Summing over fundamental discriminants $D\in\mathcal{H}_{2g+1}$, we have
\begin{equation}
    \sum_{D\in\mathcal{H}_{2g+1}}L\left(\frac{1}{2},\chi_D\right)^k=\sum_{D\in\mathcal{H}_{2g+1}}Q_k(\log_q|D|)(1+o(1)),
\end{equation}
where $Q_k$ is the polynomial of degree $\frac{1}{2}k(k+1)$ given by the $k$-fold residue 
\begin{align*}
    Q_k(x)=\frac{(-1)^{\frac{k(k-1)}{2}}2^k}{k!}\frac{1}{(2\pi i)^k}\oint\dotsc\oint\frac{G(z_1,\dotsc,z_k)\Delta(z_1^2,\dotsc,z_k^2)^2}{\prod_{j=1}^kz_j^{2k-1}}q^{\frac{x}{2}\sum_{j=1}^kz_j}dz_1\dotsc dz_k,
\end{align*}
where $\Delta(z_1,\dotsc,z_k)$ is the Vandermonde determinant defined in (\ref{eq:1.9}),
\begin{equation*}
    G(z_1,\dotsc,z_k)=A\left(\frac{1}{2};z_1,\dotsc,z_k\right)\prod_{j=1}^kX\left(\frac{1}{2}+z_j\right)^{-\frac{1}{2}}\prod_{1\leq i\leq j\leq k}\zeta_{\mathbb{A}}(1+z_i+z_j)
\end{equation*}
and $A\left(\frac{1}{2};z_1,\dotsc,z_k\right)$ is the Euler product, absolutely convergent for $|\Re(z_j)|<\frac{1}{2}$ defined by
\begin{align*}
    A\left(\frac{1}{2};z_1,\dotsc,z_k\right)&=\prod_P\prod_{1\leq i\leq j\leq k}\left(1-\frac{1}{|P|^{1+z_i+z_j}}\right)\\
    &\times \left(\frac{1}{2}\left(\prod_{j=1}^k\left(1-\frac{1}{|P|^{\frac{1}{2}+z_j}}\right)^{-1}+\prod_{j=1}^k\left(1+\frac{1}{|P|^{\frac{1}{2}+z_j}}\right)^{-1}\right)+\frac{1}{|P|}\right)\left(1+\frac{1}{|P|}\right)^{-1}. 
\end{align*}
\end{conj}
\begin{conj}
Suppose that $\alpha_k$ and $\gamma_q$ are positive and $q$ odd is the fixed cardinality of the finite field $\mathbb{F}_q$. Let $\mathfrak{D}=\{L(s,\chi_D):D\in\mathcal{H}_{2g+1}\}$ to be the family of L-functions associated with the quadratic character $\chi_D$. Then, using the same notions as in previous conjecture, we have
\begin{align*}
    &\sum_{D\in\mathcal{H}_{2g+1}}\frac{\prod_{k=1}^KL\left(\frac{1}{2}+\alpha_k,\chi_D\right)}{\prod_{q=1}^QL\left(\frac{1}{2}+\gamma_q,\chi_D\right)}\\
    &=\sum_{D\in\mathcal{H}_{2g+1}}\sum_{\epsilon\in\{-1,1\}^K}|D|^{\frac{1}{2}\sum_{k=1}^K(\epsilon_k\alpha_k-\alpha_k)}\prod_{k=1}^KX\left(\frac{1}{2}+\frac{\alpha_k-\epsilon_k\alpha_k}{2}\right)\\
    &\times Y_{\mathfrak{D}}(\epsilon_1\alpha_1,\dotsc,\epsilon_K\alpha_K;\gamma)A_{\mathfrak{D}}(\epsilon_1\alpha_1,\dotsc,\epsilon_K\alpha_K;\gamma)+o(|D|),
\end{align*}
where 
\begin{align*}
    A_{\mathfrak{D}}(\alpha;\gamma)&=\prod\frac{\prod_{1\leq j\leq k\leq K}\left(1-\frac{1}{|P|^{1+\alpha_j+\alpha_k}}\right)\prod_{1\leq q<r\leq Q}\left(1-\frac{1}{|P|^{1+\gamma_q+\gamma_r}}\right)}{\prod_{k=1}^K\prod_{q=1}^Q\left(1-\frac{1}{|P|^{1+\alpha_k+\gamma_q}}\right)}\\
    &\times \left(1+\left(1+\frac{1}{|P|}\right)^{-1}\sum_{0<\sum_ka_k+\sum_qc_q\text{ is even}}\frac{\prod_{q=1}^Q\mu(P^{c_q})}{|P|^{\sum_ka_k\left(\frac{1}{2}+\alpha_k\right)+\sum_qc_q\left(\frac{1}{2}+\gamma_q\right)}}\right)
\end{align*}
and
\begin{equation*}
    Y_{\mathfrak{D}}(\alpha;\gamma)=\frac{\prod_{1\leq j\leq k\leq K}\zeta_{\mathbb{A}}(1+\alpha_j+\alpha_k)\prod_{1\leq q<r\leq Q}\zeta_{\mathbb{A}}(1+\gamma_q+\gamma_r)}{\prod_{k=1}^K\prod_{q=1}^Q\zeta_{\mathbb{A}}(1+\alpha_k+\gamma_q)}.
\end{equation*}
If we let
\begin{align*}
    H_{D,|D|,\alpha,\gamma}(w)=|D|^{\frac{1}{2}\sum_{k=1}^Kw_k}\prod_{k=1}^KX\left(\frac{1}{2}+\frac{\alpha_k-w_k}{2}\right)Y_{\mathfrak{D}}(w_1,\dotsc,w_K;\gamma)A_{\mathfrak{D}}(\epsilon_1\alpha_1,\dotsc,\epsilon_K\alpha_K;\gamma),
\end{align*}
then the conjecture may be formulated as
\begin{align*}
    &\sum_{D\in\mathcal{H}_{2g+1}}\frac{\prod_{k=1}^KL\left(\frac{1}{2}+\alpha_k,\chi_D\right)}{\prod_{q=1}^QL\left(\frac{1}{2}+\gamma_q,\chi_D\right)}=\sum_{D\in\mathcal{H}_{2g+1}}|D|^{-\frac{1}{2}\sum_{k=1}^K\alpha_k}\sum_{\epsilon\in\{-1,1\}^K}H_{D,|D|,\alpha,\gamma}(\epsilon_1\alpha_1,\dotsc,\epsilon_K\alpha_K)+o(|D|).
\end{align*}
\end{conj}
\par\noindent
Using the Ratios conjecture Andrade and Keating obtained the one level density for the zeros of this family of quadratic Dirichlet L-functions. Bui and Florea \cite{Bui2018} studied the one level density and the pair correlation of the zeros of this family of quadratic Dirichlet L-functions. This allowed them to obtain non-vanishing results and lower bounds for the proportion of simple zeros. Also, for some restricted test functions, they compute some lower order terms which is not detected by the Ratios conjecture. In a more recent paper, Bui, Florea and Keating \cite{Bui2021} used the Ratios conjecture to compute the one and two level densities of zeros of quadratic Dirichlet L-functions in function fields. \\
\par\noindent
Florea \cite{Florea2017} improved the asymptotic formula obtained by Andrade and Keating (\ref{eq:2.3}) by obtaining a secondary main term of size $gq^{\frac{2g+1}{3}}$ whilst bounding the error term by $q^{\frac{g}{2}(1+\epsilon)}$. In particular she proved that
\begin{equation}
    \sum_{D\in\mathcal{H}_{2g+1}}L\left(\frac{1}{2},\chi_D\right)=\frac{P(1)}{2\zeta_{\mathbb{A}}(2)}q^{2g+1}\left[(2g+1)+1+\frac{4}{\log q}\frac{P'}{P}(1)\right]+q^{\frac{2g+1}{3}}R(2g+1)+O(q^{\frac{g}{2}(1+\epsilon)}),
\end{equation}
where $R$ is a polynomial of degree 1 that can be explicitly calculated. In this setting Florea \cite{Florea2017a,Florea2017c} also proved that
\begin{align}
    \sum_{D\in\mathcal{H}_{2g+1}}L\left(\frac{1}{2},\chi_D\right)^2&=q^{2g+1}R_2(2g+1)+O(q^{g(1+\epsilon)}),\\
    \sum_{D\in\mathcal{H}_{2g+1}}L\left(\frac{1}{2},\chi_D\right)^3&=q^{2g+1}R_3(2g+1)+O(q^{\frac{3g}{2}(1+\epsilon)})
\end{align}
and
\begin{equation}
    \sum_{D\in\mathcal{H}_{2g+1}}L\left(\frac{1}{2},\chi_D\right)^4=q^{2g+1}(a_{10}g^{10}+a_9g^9+a_8g^8)+O(q^{2g+1}g^{7+\frac{1}{2}+\epsilon}),
\end{equation}
where $R_2$ and $R_3$ are polynomials of degree 3 and 6 respectively and $a_{10}, a_9$ and $a_8$ are arithmetic factors. Florea also showed that the asymptotic formulas agreed with Conjecture 2.1. For the third moment, Dicanou \cite{Diaconu2019} proved the existence of a secondary main term of size $q^{\frac{3}{4}(2g+1)}$ whilst bounding the error term by $q^{(2g+1)\left(\frac{2}{3}+\delta\right)}$ for every small $\delta>0$. \\
\par\noindent
In a recent paper, Andrade, Jung and Shamesaldeen \cite{AndradeJungShamesaldeen2018} conjectured the integral moments and ratios of quadratic Dirichlet L-functions over monic irreducible polynomials in $\mathbb{F}_q[T]$ and showed that their conjecture agrees with the asymptotic formulas obtained by Andrade and Keating \cite{Andrade2013} and Bui and Florea \cite{Bui2020}.
\subsection{Quadratic Function Field in Even Characteristic}
For the remainder of this article, we assume that $q$ is a power of 2. Any separable quadratic extension $K$ of $k$ is of the form $K=K_u:=k(x_u)$, where $x_u$ is a zero of $X^2+X+u=0$ for some $u\in k$. Two elements $u,v\in k$ are equivalent if $K_u=K_v$. Furthermore, they are also equivalent if and only if $u+v=\rho(w)$ for $w\in k$, where $\rho:k\rightarrow k$ is an additive homomorphism defined by $\rho(x)=x^2+x$ (for more information, see \cite{Hasse1934,Hu2010}). For $\xi\in\mathbb{F}_q\backslash\rho(\mathbb{F}_q)$, the following Lemma is due to Y. Li, but a proof is given in \cite{Bae2018}.
\begin{lemma}[\cite{Bae2018},Lemma 2.2]
Any separable quadratic extension $K$ of $k$ is of the form $K=K_u$, where $u\in k$ can be uniquely normalised to satisfy the following conditions:
\begin{equation}\label{eq:2.9}
    u=\sum_{i=1}^m\sum_{j=1}^{e_i}\frac{Q_{i,j}}{P_i^{2j-1}}+\sum_{\ell=1}^n\alpha_{\ell}T^{2\ell-1}+\alpha,
\end{equation}
where $P_i\in\mathcal{P}$ are distinct, $Q_{i,j}\in\mathbb{A}$ with deg$(Q_{i,j})<$deg$(P_i)$, $Q_{i,e_i}\neq 0,\alpha\in\{0,\xi\}, \alpha_{\ell}\in\mathbb{F}_q$ and $\alpha_n\neq 0$ for $n>0$.
\end{lemma}
\par\noindent
Let $u\in k$ be normalised as in (\ref{eq:2.9}). The infinite prime $\infty=\left(\frac{1}{T}\right)$ splits, is inert or ramified in $K_u$ according to if $n=0$ and $\alpha=0$, $n=0$ and $\alpha=\xi$ or $n>0$. Then the field $K_u$ is called real, inert imaginary or ramified imaginary respectively. The discriminant $D_u$ of $K_u$ is given by 
\[
D_u=
\begin{cases}
\prod_{i=1}^mP_i^{2e_i}&\text{if } n=0,\\
\prod_{i=1}^mP_i^{2e_i}\left(\frac{1}{T}\right)^{2n}&\text{if }n>0.
\end{cases}
\]
By the Hurwitz genus formula (\cite{Stichtenoth1993}, Theorem III.4.12) the genus $g$ of $K_u$ is given by
\begin{equation}
    g=\frac{1}{2}\text{deg}(D_u)-1.
\end{equation}
For $M\in\mathbb{A}^+$, let $r(M)=\prod_{P|M}P$ and $t(M)=M\times r(M)$. For $P\in\mathcal{P}$, let $\mathfrak{v}_P$ be the normalised valuation at $P$, that is $\mathfrak{v}_P(M)=e$, where $P^e||M$. Let $\mathcal{B}$ be the set of monic polynomials $M$ such that $\mathfrak{v}_P(M)=0$ or odd for any $P\in\mathcal{P}$. For $M\in\mathcal{B}$, let $\ell_P=\frac{1}{2}(\mathfrak{v}_P(M)+1)$ for any $P|M$ and let
\begin{equation*}
    \tilde{M}=\prod_{P|M}P^{\ell_P}=\sqrt{t(M)}.
\end{equation*}
For a positive integer $n$, let $\mathcal{B}_n=\{M\in\mathcal{B}:\text{deg}(t(M))=2n\}$. Similarly, let $\mathcal{C}$ be the set of rational functions $\frac{D}{M}\in k$ such that $D\in\mathbb{A}$, $M\in\mathcal{B}$ and deg$(D)<$deg$(M)$ and let $\mathcal{C}_n=\left\{\frac{D}{M}:M\in\mathcal{B}_n\right\}$. Furthermore let $\mathcal{E}$ be the set of rational functions $\frac{D}{M}\in\mathcal{C}$ of the form
\begin{equation*}
    \frac{D}{M}=\sum_{P|M}\sum_{i=1}^{\ell_P}\frac{A_{P,i}}{P^{2i-1}},
\end{equation*}
where deg$(A_{P,i})<$deg$(P)$ for any $P|M$, for all $1\leq i\leq \ell_P$ and let $\mathcal{E}_n=\mathcal{E}\cap \mathcal{C}_n$. Note that for $\frac{D}{M}\in\mathcal{E}$, gcd$(D,M)=1$ if and only if $A_{P,\ell_P}\neq 0$ for all $P|M$. Let $\mathcal{F}$ be the set of rational functions $\frac{D}{M}\in\mathcal{E}$ such that $A_{P,\ell_P}\neq 0$ for all $P|M$ and $\mathcal{F}'=\{u+\xi:u\in\mathcal{F}\}$. Further, let $\mathcal{F}_n=\mathcal{F}\cap\mathcal{E}_n$ and $\mathcal{F}'_n=\{u+\xi:u\in\mathcal{F}_n\}$. Then by the normalisation in (\ref{eq:2.9}), $u\rightarrow K_u$ defines a one-to-one correspondence between $\mathcal{F}$ (respectively $\mathcal{F}'$) and the set of real (respectively inert imaginary) separable quadratic extensions of $k$. Similarly, $u\rightarrow K_u$ defines a correspondence between $\mathcal{F}_n$ (respectively $\mathcal{F}'_n$) and the set of real (respectively inert imaginary) separable quadratic extensions $K_u$ of $k$ with genus $g-1$. \\
\par\noindent
For any positive integer $s$, let $\mathcal{G}_s$ be the set of polynomials $F(T)\in\mathbb{A}$ of the form
\begin{equation}
    F(T)=\alpha+\sum_{i=1}^s\alpha_iT^{2i-1},
\end{equation}
where $\alpha\in\{0,\xi\},\alpha_i\in\mathbb{F}_q$ and $\alpha_s\neq 0$. Let $\mathcal{G}=\cup_{s\geq 1}\mathcal{G}_s$ and $\mathcal{I}=\{u+F:u\in\tilde{\mathcal{F}},F\in\mathcal{G}\}$, where $\tilde{\mathcal{F}}=\mathcal{F}\cup\mathcal{F}_0$ and $\mathcal{F}_0=\{0\}$. By the normalisation given in (\ref{eq:2.9}), we see that $w\rightarrow K_w$ defines a one-to-one correspondence between $\mathcal{I}$ and the set of ramified imaginary separable quadratic extensions of $k$.\\
\par\noindent
For any integers $r\geq 0$ and $s\geq 1$, let $\mathcal{I}_{(r,s)}=\{u+F:u\in\mathcal{F}_r,F\in\mathcal{G}_s\}$. Then, for any integer $n\geq 1$, let $\mathcal{I}_n$ be the union of all $\mathcal{I}_{(r,s)}$, where $(r,s)$ runs over all pairs of non-negative integers $s>0$ and $r+s=n$. Then $u\rightarrow K_u$ defines a correspondence between $\mathcal{I}_n$ and the set of all ramified imaginary separable quadratic extension $K_u$ of $k$ with genus $g-1$. 
\begin{lemma}[\cite{Bae2018}, Lemma 2.3]
For positive integers $n$, we have $\#\mathcal{B}_n=q^n, \#\mathcal{E}_n=q^{2n}, \#\mathcal{F}_n=\zeta_{\mathbb{A}}(2)^{-1}q^{2n}$ and $\#\mathcal{I}_n=2\zeta_{\mathbb{A}}(2)^{-1}q^{2n-1}$. 
\end{lemma}
\begin{defn}
Let $P\in\mathcal{P}$. For $u\in k$ whose denominator is not divisible by $P$, the Haase symbol $[u,P)$ with values in $\mathbb{F}_2$ is defined by
\[
[u,P)=
\begin{cases}
0&\text{if } X^2+X\equiv u(\text{mod }P)\text{ is solvable in }\mathbb{A},\\
1&\text{otherwise}.
\end{cases}
\]
\end{defn}
\begin{defn}
For $N\in\mathbb{A}$ prime to the denominator of $u$, write $N=sgn(N)\prod_{i=1}^sP_i^{e_i}$, where $P_i\in\mathcal{P}$ are distinct and $e_i\geq 1$. Then we define $[u,N)$ to be equal to $\sum_{i=1}^se_i[u,P_i)$.
\end{defn}
\begin{defn}
For $u\in k$ and $0\neq N\in\mathbb{A}$, we define the quadratic symbol $\left\{\frac{u}{N}\right\}$ by
\[
\left\{\frac{u}{N}\right\}=
\begin{cases}
(-1)^{[u,N)}&\text{if } N \text{ is prime to the denominator of } u,\\
0&\text{otherwise}.
\end{cases}
\]
\end{defn}
\begin{remark}
The symbol $[u,N)$ is additive and the quadratic symbol $\left\{\frac{u}{N}\right\}$ is multiplicative. 
\end{remark}
\begin{defn}
For the field $K_u$, the character $\chi_u$ in $\mathbb{A}^+$ is defined by $\chi_u(f)=\left\{\frac{u}{f}\right\}$. For $\Re(s)>1$, the L-function associated with the character $\chi_u$ is defined by 
\begin{equation}
    L(s,\chi_u):=\sum_{f\in\mathbb{A}^+}\frac{\chi_u(f)}{|f|^s}=\prod_P\left(1-\frac{\chi_u(P)}{|P|^s}\right)^{-1}.
\end{equation}
\end{defn}
\par\noindent
Using the change of variable $z=q^{-s}$, we have 
\begin{equation*}
    \mathcal{L}(z,\chi_u)=\sum_{f\in\mathbb{A}^+}\chi_u(f)z^{\text{deg}(f)}=\prod_P\left(1-\chi_u(P)z^{\text{deg}(P)}\right)^{-1}.
\end{equation*}
We have that $\mathcal{L}(z,\chi_u)$ has a trivial zero at $z=1$ (respectively $z=-1)$ if and only if $K_u$ is real (respectively inert imaginary). We define the completed L-function, $\mathcal{L}^*(z,\chi_u)$ as 
\[
\mathcal{L}^*(z,\chi_u)=
\begin{cases}
\mathcal{L}(z,\chi_u)&\text{if }K_u\text{ is ramified imaginary},\\
(1-z)^{-1}\mathcal{L}(z,\chi_u)&\text{if }K_u\text{ is real},\\
(1+z)^{-1}\mathcal{L}(z,\chi_u)&\text{if }K_u\text{ is inert imaginary},
\end{cases}
\]
which is a polynomial of even degree $2g$ satisfying the functional equation
\begin{equation*}
    \mathcal{L}^*(z,\chi_u)=(qz^2)^g\mathcal{L}^*((qz)^{-1},\chi_u).
\end{equation*}
\par\noindent
In \cite{Chen2008}, Chen extended the result of Hoffstein and Rosen \cite{Hoffstein1992} to even characteristic. More precisely, Chen obtained formulas for average values of L-functions associated to orders in quadratic function fields over $\mathbb{F}_q$ and then derive formulas of average class numbers of these orders. \\
\par\noindent
In \cite{Bae2018}, Bae and Jung obtained an asymptotic formula for the first moment of Dirichlet L-functions at $s\in\mathbb{C}$ with $\Re(s)\geq \frac{1}{2}$. For convenience, we will only state their result when $s=\frac{1}{2}$.
\begin{thm}
Suppose that $q$ is a power of 2. Then we have
\begin{enumerate}
    \item 
    \begin{equation}\label{eq:2.13}
    \sum_{u\in\mathcal{I}_{g+1}}L\left(\frac{1}{2},\chi_u\right)=2\frac{P(1)}{\zeta_{\mathbb{A}}(2)}q^{2g+1}\left[g+1+\frac{2}{\log q}\frac{P'}{P}(1)\right]+O(g2^{\frac{g}{2}}q^{\frac{3g}{2}}). 
    \end{equation}
    \item
    \begin{equation}
        \sum_{u\in\mathcal{F}_{g+1}}L\left(\frac{1}{2},\chi_u\right)=\frac{P(1)}{\zeta_{\mathbb{A}}(2)}q^{2g+2}\left[g+1+\zeta_{\mathbb{A}}\left(\frac{1}{2}\right)+\frac{2}{\log q}\frac{P'}{P}(1)\right]+O(2^{\frac{g}{2}}q^{\frac{3g}{2}}).
    \end{equation}
    \item 
    \begin{equation}
        \sum_{u\in\mathcal{F}'_{g+1}}L\left(\frac{1}{2},\chi_u\right)=\frac{P(1)}{\zeta_{\mathbb{A}}(2)}q^{2g+2}\left[g+1+\frac{\zeta_{\mathbb{A}}(0)}{\zeta_{\mathbb{A}}\left(\frac{1}{2}\right)}+\frac{2}{\log q}\frac{P'}{P}(1)\right]+O(2^{\frac{g}{2}}q^{\frac{3g}{2}}).
    \end{equation}
\end{enumerate}
\end{thm}
\par\noindent
When the denominator of $u$ is a monic irreducible polynomial, Andrade, Bae and Jung \cite{Andrade2016} obtained asymptotic formulas for the first and second moments of $L\left(\frac{1}{2},\chi_u\right)$ when the sum is over all $u\in\mathcal{I}_{g+1}, u\in\mathcal{F}_{g+1}$ and $u\in\mathcal{F}'_{g+1}$. 
\section{Statements of Main Results}
In this article, we develop to even characteristic the heuristic developed in \cite{Conrey2005,Conrey2008a,Andrade2014,AndradeJungShamesaldeen2018}. The main results are the following Conjectures. 
\begin{conj}
Suppose that $q$ is a power of 2 which is the fixed cardinality of the finite field $\mathbb{F}_q$ and let
\begin{equation}
    \mathcal{X}_u(s)=(q^{2g+1})^{\frac{1}{2}-s}X(s),
\end{equation}
where $X(s)=q^{-\frac{1}{2}+s}$. That is $\mathcal{X}_u(s)$ is the factor of the functional equation
\begin{equation}
    L(s,\chi_u)=\mathcal{X}_u(s)L(1-s,\chi_u).
\end{equation}
Summing over fundamental discriminants $u\in\mathcal{I}_{g+1}$, we have
\begin{equation}
    \sum_{u\in\mathcal{I}_{g+1}}L\left(\frac{1}{2},\chi_u\right)^k=\sum_{u\in\mathcal{I}_{g+1}}Q_k(2g+1)(1+o(1)),
\end{equation}
where $Q_k(x)$ is the polynomial of degree $\frac{1}{2}k(k+1)$ given by the $k$-fold residue
\begin{equation}
    Q_k(x)=\frac{(-1)^{\frac{k(k-1)}{2}}2^k}{k!}\frac{1}{(2\pi i)^k}\oint\dotsc\oint\frac{G(z_1,\dotsc,z_k)\Delta(z_1^2,\dotsc,z_k^2)^2}{\prod_{j=1}^kz_j^{2k-1}}q^{\frac{x}{2}\sum_{j=1}^kz_j}dz_1\dotsc dz_k,
\end{equation}
$\Delta(z_1,\dotsc,z_k)$ is the Vandermonde determinant given by
\begin{equation*}
    \Delta(z_1,\dotsc,z_k)=\prod_{1\leq i<j\leq k}(z_j-z_i),
\end{equation*}
\begin{equation*}
    G(z_1,\dotsc,z_k)=A\left(\frac{1}{2};z_1,\dotsc,z_k\right)\prod_{j=1}^kX\left(\frac{1}{2}+z_j\right)^{-\frac{1}{2}}\prod_{1\leq i\leq j\leq k}\zeta_{\mathbb{A}}(1+z_i+z_j)
\end{equation*}
and $A\left(\frac{1}{2};z_1,\dotsc,z_k\right)$ is the Euler product, absolutely convergent for $|\Re(z_j)|<\frac{1}{2}$ defined by
\begin{align}
    A\left(\frac{1}{2};z_1,\dotsc,z_k\right)&=\prod_P\prod_{1\leq i\leq j\leq k}\left(1-\frac{1}{|P|^{1+z_i+z_j}}\right)\nonumber\\
    &\times \left(\frac{1}{2}\left(\prod_{j=1}^k\left(1-\frac{1}{|P|^{\frac{1}{2}+z_j}}\right)^{-1}+\prod_{j=1}^k\left(1+\frac{1}{|P|^{\frac{1}{2}+z_j}}\right)^{-1}\right)+\frac{1}{|P|}\right)\left(1+\frac{1}{|P|}\right)^{-1}.
\end{align}
\end{conj}
\begin{conj}
Suppose that the real parts of $\alpha_k$ and $\gamma_q$ are positive and that $q$ is a power of 2 which is the fixed cardinality of the finite field $\mathbb{F}_q$. Let $\mathcal{U}=\{L(s,\chi_u):u\in\mathcal{I}_{g+1}\}$ be the family of L-functions associated with the quadratic character $\chi_u$. Then using the same notation as in the previous conjecture, we have
\begin{align*}
    &\sum_{u\in\mathcal{I}_{g+1}}\frac{\prod_{k=1}^KL\left(\frac{1}{2}+\alpha_k,\chi_u\right)}{\prod_{q=1}^QL\left(\frac{1}{2}+\gamma_q,\chi_u\right)}\\
    &=\sum_{u\in\mathcal{I}_{g+1}}\sum_{\epsilon\in\{-1,1\}^K}\left(q^{2g+1}\right)^{\frac{1}{2}\sum_{k=1}^K(\epsilon_k\alpha_k-\alpha_k)}\prod_{k=1}^KX\left(\frac{1}{2}+\frac{\alpha_k-\epsilon_k\alpha_k}{2}\right)\\
    &\times A_{\mathcal{U}}(\epsilon_1\alpha_1,\dotsc,\epsilon_K\alpha_K;\gamma)Y_{\mathcal{U}}(\epsilon_1\alpha_1,\dotsc,\epsilon_K\alpha_K;\gamma)+o(q^{2g+1}),
\end{align*}
where
\begin{align*}
    A_{\mathcal{U}}(\alpha;\gamma)&=\prod_P\frac{\prod_{1\leq j\leq k\leq K}\left(1-\frac{1}{|P|^{1+\alpha_j+\alpha_k}}\right)\prod_{1\leq q<r\leq Q}\left(1-\frac{1}{|P|^{1+\gamma_q+\gamma_r}}\right)}{\prod_{k=1}^K\prod_{q=1}^Q\left(1-\frac{1}{|P|^{1+\alpha_k+\gamma_q}}\right)}\\
    &\times\left(1+\left(1+\frac{1}{|P|}\right)^{-1}\sum_{0<\sum_ka_k+\sum_qc_q\text{ is even}}\frac{\prod_{q=1}^Q\mu(P^{c_q})}{|P|^{\sum_ka_k\left(\frac{1}{2}+\alpha_k\right)+\sum_qc_q\left(\frac{1}{2}+\gamma_q\right)}}\right)
\end{align*}
and
\begin{equation*}
    Y_{\mathcal{U}}(\alpha;\gamma)=\frac{\prod_{1\leq j\leq k\leq K}\zeta_{\mathbb{A}}(1+\alpha_j+\alpha_k)\prod_{1\leq q<r\leq Q}\zeta_{\mathbb{A}}(1+\gamma_q+\gamma_r)}{\prod_{k=1}^K\prod_{q=1}^Q\zeta_{\mathbb{A}}(1+\alpha_k+\gamma_q)}.
\end{equation*}
If we let
\begin{equation*}
    H_{\mathcal{I},\alpha,\gamma}(w)=\left(q^{2g+1}\right)^{\frac{1}{2}\sum_{k=1}^Kw_k}\prod_{k=1}^KX\left(\frac{1}{2}+\frac{\alpha_k-w_k}{2}\right)A_{\mathcal{U}}(w_1,\dotsc,w_K;\gamma)Y_{\mathcal{U}}(w_1,\dotsc,w_K;\gamma),
\end{equation*}
then the conjecture may be formulated as
\begin{align*}
&\sum_{u\in\mathcal{I}_{g+1}}\frac{\prod_{k=1}^KL\left(\frac{1}{2}+\alpha_k,\chi_u\right)}{\prod_{q=1}^QL\left(\frac{1}{2}+\gamma_q,\chi_u\right)}\\
&=\sum_{u\in\mathcal{I}_{g+1}}(q^{2g+1})^{-\frac{1}{2}\sum_{k=1}^K\alpha_k}\sum_{\epsilon\in\{-1,1\}^K}H_{\mathcal{I},\alpha,\gamma}(\epsilon_1\alpha_1,\dotsc,\epsilon_K\alpha_K)+o(q^{2g+1}).
\end{align*}
\end{conj}
\section{Integral Moments of L-functions in Even characteristic}
In this section, we will present the details of the recipe for conjecturing moments of the family of Dirichlet L-functions $L(s,\chi_u)$ with $u\in\mathcal{I}_{g+1}$ as $g\rightarrow\infty$, where $\mathbb{F}_q$ is a fixed finite field with $q$ a power of 2. As in \cite{Andrade2014,AndradeJungShamesaldeen2018}, we will adjust the recipe first presented in \cite{Conrey2005} to the even characteristic setting. 
\subsection{Preliminary Lemmas}
Before we present the details, we state results which will be used later in the section.
\begin{lemma}[``Approximate" Functional Equation, \cite{Bae2018} Lemma 3.1]
Let $s\in\mathbb{C}$ with $\Re(s)\geq\frac{1}{2}$, then for $u\in\mathcal{I}$, we have
\begin{equation}
    L(s,\chi_u)=\sum_{f\in\mathbb{A}^+_{\leq g}}\frac{\chi_u(f)}{|f|^s}+\mathcal{X}_u(s)\sum_{f\in\mathbb{A}^+_{\leq g-1}}\frac{\chi_u(f)}{|f|^{1-s}},
\end{equation}
where $\mathcal{X}_u(s)=q^{(1-2s)g}$. 
\end{lemma}
\begin{lemma}[\cite{Bae2018}, Lemma 3.3]
Let $L\in\mathbb{A}^+$. Given any $\epsilon>0$, we have
\begin{equation}
    \sum_{\substack{f\in\mathbb{A}^+_n\\(f,L)=1}}\phi(f)=\frac{1}{\zeta_{\mathbb{A}}(2)}q^{2n}\prod_{P|L}(1+|P|^{-1})^{-1}+O(q^{(1+\epsilon)n}).
\end{equation}
\end{lemma}
\begin{prop}[\cite{Bae2018}, Proposition 3.20]
Let $r\geq 0$ and $s\geq 1$ be integers. For any $f\in\mathbb{A}^+_d$ with $d\leq g=r+s-1$ which is not a perfect square, we have
\begin{equation*}
    \sum_{u\in\mathcal{I}_{g+1}}\chi_u(f)\ll g2^{\frac{d}{2}}q^g.
\end{equation*}
\end{prop}
\subsection{Analogies between Classical L-functions and L-functions over Function Fields}
Let $u\in\mathcal{I}_{g+1}$. For a fixed $k$, we seek to obtain as asymptotic expression for 
\begin{equation}\label{eq:4.3}
    \sum_{u\in\mathcal{I}_{g+1}}L\left(\frac{1}{2},\chi_u\right)^k,
\end{equation}
as $g\rightarrow\infty$. To achieve this, we consider a more general expression obtained by introducing small shifts, say $\alpha_1,\dotsc,\alpha_k$
\begin{equation}
    \sum_{u\in\mathcal{I}_{g+1}}L\left(\frac{1}{2}+\alpha_1,\chi_u\right)\dotsc L\left(\frac{1}{2}+\alpha_k,\chi_u\right).
\end{equation}
By introducing such shifts, hidden structures are revealed in the form of symmetries and the calculations are simplified by the removal of higher order poles. In the end, we let each $\alpha_1,\dotsc,\alpha_k$ tend to zero to recover (\ref{eq:4.3}). \\
\par\noindent
The first step to obtain the conjecture for the integral moments is to use the ``Approximate" functional equation, Lemma 4.1. Here, we note that $\mathcal{X}_u(s)$ can be written as 
\begin{equation*}
    \mathcal{X}_u(s)=(q^{2g+1})^{\frac{1}{2}-s}X(s),
\end{equation*}
where $X(s)=q^{-\frac{1}{2}+s}$ corresponds to the gamma factor in the classical quadratic Dirichlet L-functions.
\begin{lemma}
We have that 
\begin{equation*}
    \mathcal{X}_u(s)^{\frac{1}{2}}=\mathcal{X}_u(1-s)^{-\frac{1}{2}}
\end{equation*}
and
\begin{equation*}
    \mathcal{X}_u(s)\mathcal{X}_u(1-s)=1.
\end{equation*}
\end{lemma}
\begin{proof}
The proof follows directly from the definition of $\mathcal{X}_u(s)$. 
\end{proof}
\par\noindent
Consider the $Z_L$ function
\begin{equation}\label{eq:4.5}
    Z_L(s,\chi_u)=\mathcal{X}_u(s)^{-\frac{1}{2}}L(s,\chi_u).
\end{equation}
We will apply the recipe to this $Z_L$ function since it simplifies the calculations and it satisfies a symmetric functional equation given by the following Lemma.
\begin{lemma}
Let $Z_L(s,\chi_u)$ be the function defined by (\ref{eq:4.5}), then it satisfies the following functional equation
\begin{equation*}
    Z_L(s,\chi_u)=Z_L(1-s,\chi_u).
\end{equation*}
\end{lemma}
\begin{proof}
The proof follows directly from the definition of $Z_L(s,\chi_u)$ and Lemma 4.4.
\end{proof}
\par\noindent
Thus, our goal is to obtain an asymptotic formula for the $k$-shifted moment 
\begin{equation*}
    L_u(s)=\sum_{u\in\mathcal{I}_{g+1}}Z(s;\alpha_1,\dotsc,\alpha_k),
\end{equation*}
where
\begin{equation*}
    Z(s;\alpha_1,\dotsc,\alpha_k)=\prod_{j=1}^kZ_L(s+\alpha_j,\chi_u). 
\end{equation*}
Using Lemma 4.1 and Lemma 4.4, we have that
\begin{equation} \label{eq:4.6}
    Z_L(s,\chi_u)=\mathcal{X}_u(s)^{-\frac{1}{2}}\sum_{f\in\mathbb{A}^+_{\leq g}}\frac{\chi_u(f)}{|f|^s}+\mathcal{X}_u(1-s)^{-\frac{1}{2}}\sum_{f\in\mathbb{A}^+_{\leq g-1}}\frac{\chi_u(f)}{|f|^{1-s}}.
\end{equation}
\subsection{Adapting the CFKRS recipe for the Even characteristic case}
We will present the recipe which follows from \cite{AndradeJungShamesaldeen2018,Andrade2014,Conrey2005} with the necessary modifications for the family $L(s,\chi_u)$.
\begin{enumerate}
    \item We start of with a product of $k$-shifted L-functions 
    \begin{equation*}
        Z(s;\alpha_1,\dotsc,\alpha_k)=Z_L(s+\alpha_1,\chi_u)\dotsc Z_L(s+\alpha_k,\chi_u).
    \end{equation*}
    \item Replace the L-function by its ``Approximate" functional equation (\ref{eq:4.6}). Hence we obtain
    \begin{equation}
        Z\left(\frac{1}{2};\alpha_1,\dotsc,\alpha_k\right)=\sum_{\epsilon_j=\pm 1}\prod_{j=1}^k\mathcal{X}_u\left(\frac{1}{2}+\epsilon_j\alpha_j\right)^{-\frac{1}{2}}\sum_{\substack{n_1,\dotsc,n_k\\n_j\text{ monic}\\\text{deg}(n_j)\leq f(\epsilon_j)}}\frac{\chi_u(n_1\dotsc n_k)}{\prod_{j=1}^k|n_j|^{\frac{1}{2}+\epsilon_j\alpha_j}},
    \end{equation}
    where $f(1)=g$ and $f(-1)=g-1$. 
    \item Average the sign of the functional equation. Note that in this case, the $\epsilon_f$-signs of the functional equation are all equal to 1 and therefore do not produce any effect on the final result. 
    \item Replace each summand by its expected value when averaged over $\mathcal{I}_{g+1}$. For this we have the following result.
    \begin{lemma}
    Let
    \begin{equation*}
        a_m=\prod_{P|m}\left(1+\frac{1}{|P|}\right)^{-1},
    \end{equation*}
    then 
    \[
    \lim_{g\rightarrow\infty}\frac{1}{\#\mathcal{I}_{g+1}}\sum_{u\in\mathcal{I}_{g+1}}\chi_u(m)=
    \begin{cases}
    a_m&\text{if }m\text{ is a square of a polynomial},\\
    0&\text{otherwise}.
    \end{cases}
    \]
    \end{lemma}
    \begin{proof}
    We start by considering the case when $m$ is a square. For $m=\square=\ell^2$ and by the definition of $\mathcal{I}$, we have that
    \begin{equation*}
        \frac{1}{\#\mathcal{I}_{g+1}}\sum_{u\in\mathcal{I}_{g+1}}\chi_u(m=\ell^2)=\frac{1}{\#\mathcal{I}_{g+1}}\sum_{r=0}^g\sum_{u\in\mathcal{I}_{(r,g+1-r)}}\chi_u(m=\ell^2).
    \end{equation*}
    Note that $\mathcal{I}_{(0,g+1)}=\mathcal{G}_{g+1}$. For $1\leq r\leq g$, we let
    \begin{equation*}
        \mathcal{I}_M=\{v+F:v\in\mathcal{F}_M,F\in\mathcal{G}_{g+1-r}\}.
    \end{equation*}
    Then $\mathcal{I}_{(r,g+1-r)}$ is a disjoint union of $\mathcal{I}_M$'s, where $M\in\mathcal{B}_r$. Hence we have 
    \begin{align*}
        &\frac{1}{\#\mathcal{I}_{g+1}}\sum_{r=0}^g\sum_{u\in\mathcal{I}_{(r,g+1-r)}}\chi_u(m=\ell^2)\\
        &=\frac{1}{\#\mathcal{I}_{g+1}}\sum_{F\in\mathcal{G}_{g+1}}\chi_F(m=\ell^2)+\frac{1}{\#\mathcal{I}_{g+1}}\sum_{r=1}^g\sum_{M\in\mathcal{B}_r}\sum_{u\in\mathcal{I}_M}\chi_u(m=\ell^2)\\
        &=\frac{1}{\#\mathcal{I}_{g+1}}\sum_{F\in\mathcal{G}_{g+1}}1+\frac{1}{\#\mathcal{I}_{g+1}}\sum_{r=1}^g\sum_{\substack{M\in\mathcal{B}_r\\(M,\ell)=1}}\sum_{u\in\mathcal{I}_M}1.
    \end{align*}
    Using Lemma 2.4 and using the fact that $\#\mathcal{G}_{g+1}=2\zeta_{\mathbb{A}}(2)^{-1}q^{g+1}$, we have
    \begin{equation*}
        \frac{1}{\#\mathcal{I}_{g+1}}\sum_{F\in\mathcal{G}_{g+1}}1=\frac{2\zeta_{\mathbb{A}}(2)^{-1}q^{g+1}}{2\zeta_{\mathbb{A}}(2)^{-1}q^{2g+1}}=q^{-g}\rightarrow 0\text{ as }g\rightarrow\infty.
    \end{equation*}
    Also, using the fact that $\#\mathcal{I}_M=2\zeta_{\mathbb{A}}(2)^{-1}q^{g+1-r}\phi(\tilde{M})$, where $M\in\mathcal{B}_r$, we have
    \begin{equation*}
        \frac{1}{\#\mathcal{I}_{g+1}}\sum_{r=1}^g\sum_{\substack{M\in\mathcal{B}_r\\(M,\ell)=1}}\sum_{u\in\mathcal{I}_M}1=q^{-g}\sum_{r=1}^gq^{-r}\sum_{\substack{\tilde{M}\in\mathbb{A}^+_r\\(\tilde{M},\ell)=1}}\phi(\tilde{M}).
    \end{equation*}
    Invoking Lemma 4.2, we have
    \begin{align*}
     q^{-g}\sum_{r=1}^gq^{-r}\sum_{\substack{\tilde{M}\in\mathbb{A}^+_r\\(\tilde{M},\ell)=1}}\phi(\tilde{M})&=\frac{q^{-g}}{\zeta_{\mathbb{A}}(2)}\prod_{P|\ell}\left(1+\frac{1}{|P|}\right)^{-1}\sum_{r=1}^gq^r+O\left(q^{-g}\sum_{r=1}^gq^{\epsilon r}\right)\\
     &=\frac{q^{-g}}{\zeta_{\mathbb{A}}(2)}\prod_{P|\ell}\left(1+\frac{1}{|P|}\right)^{-1}\frac{q}{q-1}(q^g-1)+O(q^{-g(1-\epsilon)}).
    \end{align*}
    As $g\rightarrow\infty$, the main term becomes $a_m$ and the second and error terms tend to zero. \\
    \par\noindent
    For $m$ not a perfect square, we have, by Lemma 4.3
    \begin{equation*}
        \frac{1}{\#\mathcal{I}_{g+1}}\sum_{u\in\mathcal{I}_{g+1}}\chi_u(m)\ll g2^{\frac{g}{2}}q^{-g}\rightarrow 0\text{ as }g\rightarrow\infty.
    \end{equation*}
    \end{proof}
    \par\noindent
    Using Lemma 4.6, we have that
    \begin{align*}
    \lim_{g\rightarrow\infty}\frac{1}{\#\mathcal{I}_{g+1}}\sum_{u\in\mathcal{I}_{g+1}}\sum_{\substack{n_1,\dotsc,n_k\\n_j\text{ monic}}}\frac{\chi_u(n_1\dotsc n_k)}{\prod_{j=1}^k|n_j|^{\frac{1}{2}+\epsilon_j\alpha_j}}&=\sum_{\substack{n_1,\dotsc,n_k\\n_j\text{ monic}\\n_1\dotsc n_k=m^2}}\frac{a_{m^2}}{\prod_{j=1}^k|n_j|^{\frac{1}{2}+\epsilon_j\alpha_j}}\\
    &=\sum_{m\text{ monic}}\sum_{\substack{n_1,\dotsc,n_k\\n_j\text{ monic}\\n_1\dotsc n_k=m^2}}\frac{a_{m^2}}{\prod_{j=1}^k|n_j|^{\frac{1}{2}+\epsilon_j\alpha_j}}.
    \end{align*}
    \item If we let 
    \begin{equation*}
        R_k\left(\frac{1}{2};\epsilon_1\alpha_1,\dotsc,\epsilon_k\alpha_k\right)=\sum_{m\text{ monic}}\sum_{\substack{n_1,\dotsc,n_k\\n_j\text{ monic}\\n_1\dotsc n_k=m^2}}\frac{a_{m^2}}{\prod_{j=1}^k|n_j|^{\frac{1}{2}+\epsilon_j\alpha_j}},
    \end{equation*}
    then the extended sum produced by the recipe is 
    \begin{equation*}
        M\left(\frac{1}{2};\alpha_1,\dotsc,\alpha_k\right)=\sum_{\epsilon_j=\pm 1}\prod_{j=1}^k\mathcal{X}_u\left(\frac{1}{2}+\epsilon_j\alpha_j\right)^{-\frac{1}{2}}R_k\left(\frac{1}{2};\epsilon_1\alpha_1,\dotsc,\epsilon_k\alpha_k\right).
    \end{equation*}
    \item 
    Then the conclusion is 
    \begin{equation*}
        \sum_{u\in\mathcal{I}_{g+1}}Z_L\left(\frac{1}{2};\alpha_1,\dotsc,\alpha_k\right)=\sum_{u\in\mathcal{I}_{g+1}}M\left(\frac{1}{2};\alpha_1,\dotsc,\alpha_k\right)(1+o(1)).
    \end{equation*}
\end{enumerate}
\subsection{Putting the Conjecture into a more useful form}
The conjecture is problematic in the form presented because of the individual terms have poles that cancel when summed. In this subsection, we put the conjecture into a more useful form, writing $R_k$ as an Euler product and then factoring out the appropriate $\zeta_{\mathbb{A}}(s)$-factors. \\
\par\noindent
First note that $a_m$ is multiplicative since
\begin{equation*}
    a_{mn}=a_ma_n\text{ whenever gcd}(m,n)=1,
\end{equation*}
where 
\begin{equation*}
    a_m=\prod_{P|m}\left(1+\frac{1}{|P|}\right)^{-1}.
\end{equation*}
Define
\begin{equation*}
    \psi(x):=\sum_{\substack{n_1,\dotsc,n_k\\n_j\text{ monic}\\n_1\dotsc n_k=x}}\frac{1}{|n_1|^{s+\alpha_1}\dotsc|n_k|^{s+\alpha_k}},
\end{equation*}
so that $\psi(m^2)$ is multiplicative on $m$. Therefore
\begin{align*}
    \sum_{m\text{ monic}}\sum_{\substack{n_1,\dotsc,n_k\\n_j\text{ monic}\\n_1\dotsc n_k=m^2}}\frac{a_{m^2}}{|n_1|^{s+\alpha_1}\dotsc|n_k|^{s+\alpha_k}}&=\sum_{m\text{ monic}}a_{m^2}\sum_{\substack{n_1,\dotsc,n_k\\n_j\text{ monic}\\n_1\dotsc n_k=m^2}}\frac{1}{|n_1|^{s+\alpha_1}\dotsc|n_k|^{s+\alpha_k}}\\
&=\sum_{m\text{ monic}}a_{m^2}\psi(m^2)\\
&=\prod_P\left(1+\sum_{j=1}^{\infty}a_{p^{2j}}\psi(P^{2j})\right),
\end{align*}
where 
\begin{equation}\label{eq:4.8}
        \psi(P^{2j})=\sum_{\substack{n_1,\dotsc,n_k\\n_j\text{ monic}\\n_1\dotsc n_k=P^{2j}}}\frac{1}{|n_1|^{s+\alpha_1}\dotsc|n_k|^{s+\alpha_k}}.
\end{equation}
Since we have $n_1\dotsc n_k=P^{2j}$, then for each $i=1,\dotsc,k$, write $n_i=P^{e_i}$ for some $e_i\geq 0$ and $e_1+\dotsc+e_k=2j$, (\ref{eq:4.8}) becomes
\begin{equation*}
    \psi(P^{2j})=\sum_{\substack{e_1,\dotsc,e_k\geq 0\\e_1+\dotsc+e_k=2j}}\prod_{i=1}^k\frac{1}{|P|^{e_i(s+\alpha_i)}}.
\end{equation*}
Therefore we have
\begin{align}\label{eq:4.9}
    R_k(s;\alpha_1,\dotsc,\alpha_k)&=\prod_P\left(1+\sum_{j=1}^{\infty}a_{P^{2j}}\psi(P^{2j})\right)\nonumber\\
    &=\prod_P\left(1+\sum_{j=1}^{\infty}a_{P^{2j}}\sum_{\substack{e_1,\dotsc,e_k\geq 0\\e_1+\dotsc+e_k=2j}}\prod_{i=1}^k\frac{1}{|P|^{e_i(s+\alpha_i)}}\right).
\end{align}
However, we know that 
\begin{equation*}
    a_{P^{2j}}=(1+|P|^{-1})^{-1},
\end{equation*}
thus (\ref{eq:4.9}) becomes 
\begin{align}\label{eq:4.10}
    R_k(s;\alpha_1,\dotsc,\alpha_k)&=\prod_P\left(1+(1+|P|^{-1})^{-1}\sum_{j=1}^{\infty}\sum_{\substack{e_1,\dotsc,e_k\geq 0\\e_1+\dotsc+
    e_k=2j}}\prod_{j=1}^k\frac{1}{|P|^{e_i(s+\alpha_i)}}\right)\nonumber\\
    &=\prod_PR_{k,P}.
\end{align}
Using the fact that 
\begin{equation*}
    (1+|P|^{-1})^{-1}=\sum_{\ell=0}^{\infty}\frac{(-1)^{\ell}}{|P|^{\ell}},
\end{equation*}
we have that 
\begin{equation*}
    R_{k,P}=1+\sum_{\ell=0}^{\infty}\sum_{j=1}^{\infty}\sum_{\substack{e_1,\dotsc,e_k\geq 0\\e_1+\dotsc+e_k=2j}}\prod_{i=1}^k\frac{(-1)^{\ell}}{|P|^{e_i(s+\alpha_i)+\ell}}.
\end{equation*}
When $\alpha_i=0$ and $s=\frac{1}{2}$, only terms with $e_1+\dotsc+e_k=2$ give rise to poles. Isolating the term with $\ell=0$ and $j=1$, we have
\begin{align*}
    R_{k,P}&=1+\sum_{e_1+\dotsc+e_k=2}\frac{1}{|P|^{e_i(s+\alpha_i)}}+(\text{lower order terms})\\
    &=1+\sum_{1\leq i\leq j\leq k}\frac{1}{|P|^{2s+\alpha_i+\alpha_j}}+(\text{lower order terms}).
\end{align*}
Hence we can write, for $\Re(\alpha_i)$ sufficiently small
\begin{equation*}
    R_{k,P}=1+\sum_{1\leq i\leq j\leq k}\frac{1}{|P|^{2s+\alpha_i+\alpha_j}}+O(|P|^{-1-2s+\epsilon})+O(|P|^{-3s+\epsilon}).
\end{equation*}
Expressing $R_{k,P}$ as a Euler product, we have 
\begin{equation*}
    R_{k,P}=\prod_{1\leq i\leq j\leq k}\left(1+\frac{1}{|P|^{2s+\alpha_i+\alpha_j}}\right)\left(1+O(|P|^{-1-2s+\epsilon})+O(|P|^{-3s+\epsilon})\right).
\end{equation*}
Since 
\begin{equation*}
    \prod_P\left(1+\frac{1}{|P|^{2s}}\right)=\frac{\zeta_{\mathbb{A}}(2s)}{\zeta_{\mathbb{A}}(4s)}
\end{equation*}
has a simple pole at $s=\frac{1}{2}$ and
\begin{equation*}
    \prod_P\left(1+O(|P|^{-1-2s+\epsilon})+O(|P|^{-3s+\epsilon}\right)
\end{equation*}
is analytic in $\Re(s)>\frac{1}{3}$, we see that $\prod_PR_{k,P}$ has a pole at $s=\frac{1}{2}$ of order $\frac{1}{2}k(k+1)$ if $\alpha_1=\dotsc=\alpha_k=0$. Thus we write 
\begin{equation*}
    R_k(s;\alpha_1,\dotsc,\alpha_k)=\prod_{1\leq i \leq j\leq k}\zeta_{\mathbb{A}}(2s+\alpha_i+\alpha_j)A(s;\alpha_1,\dotsc,\alpha_k),
\end{equation*}
where 
\begin{equation}\label{eq:4.11}
    A(s;\alpha_1,\dotsc,\alpha_k)=\prod_P\left(R_{k,P}(s;\alpha_1,\dotsc,\alpha_k)\prod_{1\leq i\leq j\leq k}\left(1-\frac{1}{|P|^{2s+\alpha_i+\alpha_j}}\right)\right).
\end{equation}
Here $A(s,\alpha_1,\dotsc,\alpha_k)$ defines an absolutely convergent Dirichlet series for $\Re(s)=\frac{1}{2}$ and for all $\alpha_j$'s positive. We can further write $A(s,\alpha_1,\dotsc,\alpha_k)$ by the following Lemma.
\begin{lemma}
Using the notation stated previously, we have
\begin{align}
    A\left(\frac{1}{2};z_1,\dotsc,z_k\right)&=\prod_P\prod_{1\leq i\leq j\leq k}\left(1-\frac{1}{|P|^{1+z_i+z_j}}\right)\nonumber\\
    &\times\left(\frac{1}{2}\left(\prod_{j=1}^k\left(1-\frac{1}{|P|^{\frac{1}{2}+z_j}}\right)^{-1}+\prod_{j=1}^k\left(1+\frac{1}{|P|^{\frac{1}{2}+z_j}}\right)^{-1}\right)+\frac{1}{|P|}\right)\left(1+\frac{1}{|P|}\right)^{-1}.
\end{align}
\end{lemma}
\begin{proof}
From (\ref{eq:4.11}), we have that
\begin{equation*}
    A\left(\frac{1}{2};z_1,\dotsc,z_k\right)=\prod_P\left(R_{k,P}\left(\frac{1}{2},z_1,\dotsc,z_k\right)\prod_{1\leq i\leq j\leq k}\left(1-\frac{1}{|P|^{1+z_i+z_j}}\right)\right),
\end{equation*}
where $R_{k,P}$ is defined in (\ref{eq:4.10}). Further, we have
\begin{equation*}
    R_{k,P}=\left(1+|P|^{-1}\right)^{-1}\left(\left(1+|P|^{-1}\right)+\sum_{j=1}^{\infty}\sum_{\substack{e_1,\dotsc,e_k\geq 0\\e_1+\dotsc+e_k=2j}}\prod_{i=1}^k\frac{1}{|P|^{e_i\left(\frac{1}{2}+z_i\right)}}\right).
\end{equation*}
Thus
\begin{align*}
    1+\sum_{j=1}^{\infty}\sum_{\substack{e_1,\dotsc,e_k\geq 0\\e_1+\dotsc+e_k=2j}}\prod_{i=1}^k\frac{1}{|P|^{e_i\left(\frac{1}{2}+z_i\right)}}&=\frac{1}{2}\sum_{j=0}^{\infty}\sum_{\substack{e_1,\dotsc,e_k\geq 0\\e_1+\dotsc+e_k=2j}}2\prod_{i=1}^k\frac{1}{|P|^{e_i\left(\frac{1}{2}+z_i\right)}}\\
    &=\frac{1}{2}\left(\prod_{i=1}^k\sum_{e_i=0}^{\infty}\left(\frac{1}{|P|^{\frac{1}{2}+z_i}}\right)^{e_i}+\prod_{i=1}^k\sum_{e_i=0}^{\infty}(-1)^{e_1+\dotsc+e_k}\left(\frac{1}{|P|^{\frac{1}{2}+z_i}}\right)^{e_i}\right)\\
    &=\frac{1}{2}\left(\prod_{i=1}^k\left(1-\frac{1}{|P|^{\frac{1}{2}+z_i}}\right)^{-1}+\prod_{i=1}^k\left(1+\frac{1}{|P|^{\frac{1}{2}+z_i}}\right)^{-1}\right).
\end{align*}
\end{proof}
\par\noindent
We have 
\begin{align*}
    M\left(\frac{1}{2};\alpha_1,\dotsc,\alpha_k\right)&=\sum_{\epsilon_j=\pm 1}\prod_{j=1}^k\mathcal{X}_u\left(\frac{1}{2}+\epsilon_j\alpha_j\right)^{-\frac{1}{2}}R_k\left(\frac{1}{2};\epsilon_1\alpha_1,\dotsc,\epsilon_k\alpha_k\right)\\
    &=\sum_{\epsilon_j=\pm 1}\prod_{j=1}^k\mathcal{X}_u\left(\frac{1}{2}+\epsilon_j\alpha_j\right)^{-\frac{1}{2}}\prod_{1\leq i\leq j\leq k}\zeta_{\mathbb{A}}(1+\epsilon_i\alpha_i+\epsilon_j\alpha_j)A\left(\frac{1}{2};\epsilon_1\alpha_1,\dotsc,\epsilon_k\alpha_k\right).
\end{align*}
Therefore the conjecture takes the form
\begin{align*}
    &\sum_{u\in\mathcal{I}_{g+1}}Z_L\left(\frac{1}{2},\alpha_1,\dotsc,\alpha_k\right)\\&=\sum_{u\in\mathcal{I}_{g+1}}\sum_{\epsilon_j=\pm 1}\prod_{j=1}^k\mathcal{X}_u\left(\frac{1}{2}+\epsilon_j\alpha_j\right)^{-\frac{1}{2}}A\left(\frac{1}{2};\epsilon_1\alpha_1,\dotsc,\epsilon_k\alpha_k\right)\prod_{1\leq i\leq j\leq k}\zeta_{\mathbb{A}}(1+\epsilon_i\alpha_i+\epsilon_j\alpha_j)(1+o(1)).
\end{align*}
Using the fact that
\begin{equation*}
    \mathcal{X}_u(s)=q^{g(1-2s)}X(s),
\end{equation*}
where $X(s)=q^{-\frac{1}{2}+s}$, then 
\begin{equation*}
    \mathcal{X}_u\left(\frac{1}{2}+\epsilon_j\alpha_j\right)^{-\frac{1}{2}}=(q^{2g+1})^{\frac{\epsilon_j\alpha_j}{2}}X\left(\frac{1}{2}+\epsilon_j\alpha_j\right)^{-\frac{1}{2}}.
\end{equation*}
Thus we arrive at the following conjecture:
\begin{align}\label{eq:4.13}
    &\sum_{u\in\mathcal{I}_{g+1}}Z_L\left(\frac{1}{2}+\alpha_1,\chi_u\right)\dotsc Z_L\left(\frac{1}{2}+\alpha_k,\chi_u\right)\nonumber\\
    &=\sum_{\epsilon_j=\pm 1}\prod_{j=1}^kX\left(\frac{1}{2}+\epsilon_j\alpha_j\right)^{-\frac{1}{2}}\sum_{u\in\mathcal{I}_{g+1}}(q^{2g+1})^{\frac{1}{2}\sum_{j=1}^k\epsilon_j\alpha_j}R_k\left(\frac{1}{2};\epsilon_1\alpha_1,\dotsc,\epsilon_k\alpha_k\right)(1+o(1)).
\end{align}
\subsection{The Contour Integral Representation of the Conjecture}
The following Lemma is due to Conrey, Farmer, Keating, Rubinstein and Snaith \cite{Conrey2005}. 
\begin{lemma}
Suppose $F$ is a symmetric function in $k$ variables, regular near $(0,\dotsc,0)$ and that $f(s)$ has a simple pole of residue 1 at $s=0$ and is otherwise analytic in a neighbourhood of $s=0$ amd let
\begin{equation*}
    K(\alpha_1,\dotsc,\alpha_k)=F(\alpha_1,\dotsc,\alpha_k)\prod_{1\leq i\leq j\leq k}f(\alpha_i+\alpha_j),
\end{equation*}
or
\begin{equation*}
    K(\alpha_1,\dotsc,\alpha_k)=F(\alpha_1,\dotsc,\alpha_k)\prod_{1\leq i<j\leq k}f(\alpha_i+\alpha_j).
\end{equation*}
If $\alpha_i+\alpha_j$ are contained in the region of analyticity of $f(s)$, then 
\begin{align*}
    \sum_{\epsilon_j=\pm 1}K(\epsilon_1\alpha_1,\dotsc,\epsilon_k\alpha_k)&=\frac{(-1)^{\frac{k(k-1)}{2}}2^k}{k!}\frac{1}{(2\pi i)^k}\oint\dotsc\oint K(z_1,\dotsc,z_k)\\
    &\times\frac{\Delta(z_1^2,\dotsc,z_k^2)^2\prod_{j=1}^kz_j}{\prod_{i=1}^k\prod_{j=1}^k(z_i-\alpha_j)(z_i+\alpha_j)}dz_1\dotsc dz_k
\end{align*}
and 
\begin{align*}
    \sum_{\epsilon_j=\pm 1}\left(\prod_{j=1}^k\epsilon_j\right)K(\epsilon_1\alpha_1,\dotsc,\epsilon_k\alpha_k)&=\frac{(-1)^{\frac{k(k-1)}{2}}2^k}{k!}\frac{1}{(2\pi i)^k}\oint\dotsc\oint K(z_1,\dotsc,z_k)\\
    &\times\frac{\Delta(z_1^2,\dotsc,z_k^2)^2\prod_{j=1}^k\alpha_j}{\prod_{i=1}^k\prod_{j=1}^k(z_i-\alpha_j)(z_i+\alpha_j)}dz_1\dotsc dz_k,
\end{align*}
where the path of integration encloses the $\pm\alpha_j$'s.
\end{lemma}
\par \noindent
We will use Lemma 4.8 to write Conjecture (\ref{eq:4.13}) as a contour integral. First note that
\begin{equation*}
    \sum_{u\in\mathcal{I}_{g+1}}\prod_{j=1}^kZ_L\left(\frac{1}{2}+\alpha_j,\chi_u\right)=\sum_{u\in\mathcal{I}_{g+1}}\prod_{j=1}^k\mathcal{X}_u\left(\frac{1}{2}+\alpha_j\right)^{-\frac{1}{2}}L\left(\frac{1}{2}+\alpha_j,\chi_u\right)
\end{equation*}
and we can factor out the $\mathcal{X}_u\left(\frac{1}{2}+\alpha_j\right)^{-\frac{1}{2}}$'s leaving
\begin{align*}
    &\sum_{u\in\mathcal{I}_{g+1}}\prod_{j=1}^kL\left(\frac{1}{2}+\alpha_j,\chi_u\right)\\
    &=\sum_{u\in\mathcal{I}_{g+1}}\prod_{j=1}^k\mathcal{X}_u\left(\frac{1}{2}+\alpha_j\right)^{\frac{1}{2}}\sum_{\epsilon_j=\pm 1}\mathcal{X}_u\left(\frac{1}{2}+\epsilon_j\alpha_j\right)^{-\frac{1}{2}}A\left(\frac{1}{2};\epsilon_1\alpha_1,\dotsc,\epsilon_k\alpha_k\right)\\
    &\times\prod_{1\leq i\leq j\leq k}\zeta_{\mathbb{A}}(1+\epsilon_i\alpha_i+\epsilon_j\alpha_j)(1+o(1)).
\end{align*}
Using the definition of $\mathcal{X}_u(s)$, we have
\begin{align*}
    &\sum_{u\in\mathcal{I}_{g+1}}\prod_{j=1}^kL\left(\frac{1}{2}+\alpha_j,\chi_u\right)\\
    &=\sum_{u\in\mathcal{I}_{g+1}}\prod_{j=1}^k(q^{2g+1})^{-\frac{1}{2}\sum_{j=1}^k\alpha_j}X\left(\frac{1}{2}+\alpha_j\right)^{\frac{1}{2}}\sum_{\epsilon_j=\pm 1}(q^{2g+1})^{\frac{1}{2}\sum_{j=1}^k\epsilon_j\alpha_j}X\left(\frac{1}{2}+\epsilon_j\alpha_j\right)^{-\frac{1}{2}}\\
    &\times A\left(\frac{1}{2};\epsilon_1\alpha_1,\dotsc,\epsilon_k\alpha_k\right)\prod_{1\leq i\leq j\leq k}\zeta_{\mathbb{A}}(1+\epsilon_i\alpha_i+\epsilon_j\alpha_j)(1+o(1)).
\end{align*}
Taking out a factor of $\log q $ from each term in the second product, we have
\begin{align}\label{eq:4.14}
    &\sum_{u\in\mathcal{I}_{g+1}}L\left(\frac{1}{2}+\alpha_1,\chi_u\right)\dotsc L\left(\frac{1}{2}+\alpha_k,\chi_u\right)\nonumber\\
    &=\sum_{u\in\mathcal{I}_{g+1}}\frac{\prod_{j=1}^k(q^{2g+1})^{-\frac{1}{2}\sum_{j=1}^k\alpha_j}X\left(\frac{1}{2}+\alpha_j\right)^{\frac{1}{2}}}{(\log q)^{\frac{k(k+1)}{2}}}\sum_{\epsilon_j=\pm 1}(q^{2g+1})^{\frac{1}{2}\sum_{j=1}^k\epsilon_j\alpha_j}X\left(\frac{1}{2}+\epsilon_j\alpha_j\right)^{-\frac{1}{2}}\nonumber\\
    &\times A\left(\frac{1}{2};\epsilon_1\alpha_1,\dotsc,\epsilon_k\alpha_k\right)\prod_{1\leq i\leq j\leq k}\zeta_{\mathbb{A}}(1+\epsilon_i\alpha_i+\epsilon_j\alpha_j)(\log q)(1+o(1)).
\end{align}
If we call
\begin{equation*}
    F(\alpha_1,\dotsc,\alpha_k)=\prod_{j=1}^k(q^{2g+1})^{\frac{1}{2}\sum_{j=1}^k\alpha_j}X\left(\frac{1}{2}+\alpha_j\right)^{-\frac{1}{2}}A\left(\frac{1}{2};\alpha_1,\dotsc,\alpha_k\right)
\end{equation*}
and
\begin{equation}\label{eq:4.15}
    f(s)=\zeta_{\mathbb{A}}(1+s)(\log q)\text{ and so }f(\alpha_i+\alpha_j)=\zeta_{\mathbb{A}}(1+\alpha_i+\alpha_j)(\log q),
\end{equation}
then $f(s)$ has a simple pole at $s=0$ with residue 1. Denoting
\begin{equation}\label{eq:4.16}
    K(\alpha_1,\dotsc,\alpha_k)=F(\alpha_1,\dotsc,\alpha_k)\prod_{1\leq i\leq j\leq k}f(\alpha_i+\alpha_j),
\end{equation}
then (\ref{eq:4.14}) is equal to 
\begin{equation}\label{eq:4.17}
    \left(\sum_{u\in\mathcal{I}_{g+1}}\frac{\prod_{j=1}^k(q^{2g+1})^{-\frac{1}{2}\sum_{j=1}^k\alpha_j}X\left(\frac{1}{2}+\alpha_j\right)^{\frac{1}{2}}}{(\log q)^{\frac{k(k+1)}{2}}}\sum_{\epsilon_j=\pm 1}K(\epsilon_1\alpha_1,\dotsc,\epsilon_k\alpha_k)\right)(1+o(1)).
\end{equation}
Using Lemma 4.8, (\ref{eq:4.17}) is equal to
\begin{align}\label{eq:4.18}
    &\sum_{u\in\mathcal{I}_{g+1}}\frac{(q^{2g+1})^{-\frac{1}{2}\sum_{j=1}^k\alpha_j}X\left(\frac{1}{2}+\alpha_j\right)^{\frac{1}{2}}}{(\log q)^{\frac{k(k+1)}{2}}}\frac{(-1)^{\frac{k(k-1)}{2}}2^k}{k!}\oint\dotsc\oint K(z_1,\dotsc,z_k)\nonumber\\
    &\times\frac{\Delta(z_1^2,\dotsc,z_k^2)^2\prod_{j=1}^kz_j}{\prod_{i=1}^k\prod_{j=1}^k(z_i+\alpha_j)(z_i+\alpha_j)}dz_1\dotsc dz_k+o(q^{2g+1}).
\end{align}
Using (\ref{eq:4.16}) and (\ref{eq:4.15}), we have that (\ref{eq:4.18}) is equal to
\begin{align*}
    &\sum_{u\in\mathcal{I}_{g+1}}(q^{2g+1})^{-\frac{1}{2}\sum_{j=1}^k\alpha_j}X\left(\frac{1}{2}+\alpha_j\right)^{\frac{1}{2}}\frac{(-1)^{\frac{k(k-1)}{2}}2^k}{k!}\frac{1}{(2\pi i)^k}\oint\dotsc\oint F(z_1,\dotsc,z_k)\\
    &\times \prod_{1\leq i\leq j\leq k}\zeta_{\mathbb{A}}(1+z_i+z_j)\frac{\Delta(z_1^2,\dotsc,z_k^2)^2\prod_{j=1}^kz_j}{\prod_{i=1}^k\prod_{j=1}^k(z_i-\alpha_j)(z_i+\alpha_j)}dz_1\dotsc dz_k+o(q^{2g+1}).
\end{align*}
Let
\begin{equation*}
    G(z_1,\dotsc,z_k)=\prod_{j=1}^kX\left(\frac{1}{2}+z_j\right)^{-\frac{1}{2}}A\left(\frac{1}{2};z_1,\dotsc,z_k\right)\prod_{1\leq i\leq j\leq k}\zeta_{\mathbb{A}}(1+z_i+z_j),
\end{equation*}
then we have that
\begin{align*}
   & \sum_{u\in\mathcal{I}_{g+1}}L\left(\frac{1}{2}+\alpha_1,\chi_u\right)\dotsc L\left(\frac{1}{2}+\alpha_k,\chi_u\right)\\
   &=\sum_{u\in\mathcal{I}_{g+1}}\prod_{j=1}^k(q^{2g+1})^{-\frac{1}{2}\sum_{j=1}^k\alpha_j}X\left(\frac{1}{2}+\alpha_j\right)^{\frac{1}{2}}\frac{(-1)^{\frac{k(k-1)}{2}}2^k}{k!}\frac{1}{(2\pi i)^k}\oint\dotsc\oint G(z_1,\dotsc,z_k)\\
   &\times (q^{2g+1})^{\frac{1}{2}\sum_{j=1}^kz_j}\frac{\Delta(z_1^2,\dotsc,z_k^2)^2\prod_{j=1}^kz_j}{\prod_{i=1}^k\prod_{j=1}^k(z_i-\alpha_j)(z_i+\alpha_j)}dz_1\dotsc dz_k+o(q^{2g+1}).
\end{align*}
If we let
\begin{align*}
    Q_k(x)=\frac{(-1)^{\frac{k(k-1)}{2}}2^k}{k!}\frac{1}{(2\pi i)^k}\oint\dotsc\oint G(z_1,\dotsc,z_k)q^{\frac{x}{2}\sum_{j=1}^kz_j}\frac{\Delta(z_1^2,\dotsc,z_k^2)^2\prod_{j=1}^kz_j}{\prod_{i=1}^k\prod_{j=1}^k(z_i-\alpha_j)(z_i+\alpha_j)}dz_1\dotsc dz_k,
\end{align*}
then setting $\alpha_i=0$, we obtain the formulae stated in Conjecture 3.1.
\section{Some Conjectural Formulae for Moments of L-functions in Even characteristic}
In this section, we use Conjecture 3.1 to obtain explicit conjectural formulae for the first few moments of quadratic Dirichlet L-functions in even characteristic. 
\subsection{First Moment}
We will use Conjecture 3.1 to determine the asymptotic formula for the first moment of our family of Dirichlet L-functions and compare it with (\ref{eq:2.13}). For the first moment, Conjecture 3.1 states that
\begin{equation*}
    \sum_{u\in\mathcal{I}_{g+1}}L\left(\frac{1}{2},\chi_u\right)=\sum_{u\in\mathcal{I}_{g+1}}Q_1(2g+1)(1+o(1)),
\end{equation*}
where $Q_1(x)$ is a polynomial of degree 1. From Conjecture 3.1, we have
\begin{equation}\label{eq:5.1}
    Q_1(x)=\frac{1}{\pi i}\oint\frac{G(z_1)\Delta(z_1^2)^2q^{\frac{x}{2}z_1}}{z_1}dz_1,
\end{equation}
where
\begin{equation*}
    G(z_1)=A\left(\frac{1}{2};z_1\right)X\left(\frac{1}{2}+z_1\right)^{-\frac{1}{2}}\zeta_{\mathbb{A}}(1+2z_1).
\end{equation*}
From the definition of the Vandermonde determinant and the definition of $X(s)$, we have that
\begin{equation*}
    \Delta(z_1^2)^2=1\text{ and }X\left(\frac{1}{2}+z_j\right)^{-\frac{1}{2}}=q^{-\frac{z_1}{2}}.
\end{equation*}
Therefore (\ref{eq:5.1}) becomes
\begin{equation}\label{eq:5.2}
    Q_1(x)=\frac{1}{\pi i}\oint\frac{A\left(\frac{1}{2};z_1\right)\zeta_{\mathbb{A}}(1+2z_1)q^{\frac{x}{2}z_1}q^{-\frac{z_1}{2}}}{z_1}dz_1.
\end{equation}
From Lemma 4.7, we have
\begin{equation*}
    A\left(\frac{1}{2};z_1\right)=\prod_P\left(1-\frac{1}{|P|^{1+2z_1}}\right)\left(\frac{1}{2}\left(\left(1-\frac{1}{|P|^{\frac{1}{2}+z_1}}\right)^{-1}+\left(1+\frac{1}{|P|^{\frac{1}{2}+z_1}}\right)^{-1}\right)+\frac{1}{|P|}\right)\left(1+\frac{1}{|P|}\right)^{-1}.
\end{equation*}
We want to compute the integral (\ref{eq:5.2}) where the contour is a small circle around the origin. For this we need to locate the poles of the integrand. Let
\begin{equation}\label{eq:5.3}
    f(z_1)=\frac{A\left(\frac{1}{2};z_1\right)\zeta_{\mathbb{A}}(1+2z_1)q^{\frac{x}{2}z_1}q^{-\frac{z_1}{2}}}{z_1},
\end{equation}
then $f(z_1)$ has a double pole at $z_1=0$. To compute the residue, we expand $f(z_1)$ as a Laurent series and pick up the coefficients of $z_1^{-1}$. Expanding the numerator of (\ref{eq:5.3}) around $z_1=0$, we have
\begin{equation*}
    A\left(\frac{1}{2};z_1\right)=A\left(\frac{1}{2};0\right)+A'\left(\frac{1}{2};0\right)z_1+\frac{1}{2}A''\left(\frac{1}{2};0\right)z_1^2+\dotsc,
\end{equation*}
\begin{equation*}
    \zeta_{\mathbb{A}}(1+2z_1)=\frac{1}{2\log q}\frac{1}{z_1}+\frac{1}{2}+\frac{1}{6}(\log q)z_1-\frac{1}{90}(\log q)^3z_1^3+\dotsc,
\end{equation*}
\begin{equation*}
    q^{\frac{x}{2}z_1}=1+\frac{1}{2}(\log q)xz_1+\frac{1}{8}(\log q)^2x^2z_1^2+\dotsc
\end{equation*}
and
\begin{equation*}
    q^{-\frac{z_1}{2}}=1-\frac{1}{2}(\log q)z_1+\frac{1}{8}(\log q)^2z_1^2+\dotsc. 
\end{equation*}
Thus we have
\begin{align*}
    f(z_1)&=\left(A\left(\frac{1}{2};0\right)+A'\left(\frac{1}{2};0\right)z_1+\frac{1}{2}A''\left(\frac{1}{2};0\right)z_1^2+\dotsc\right)\left(1-\frac{1}{2}(\log q)z_1+\frac{1}{8}(\log q)^2z_1^2+\dotsc\right)\\
    &\times \left(\frac{1}{2\log q}\frac{1}{z_1}+\frac{1}{2}+\frac{1}{6}(\log q)z_1-\frac{1}{90}(\log q)^3z_1^3+\dotsc\right)\left(1+\frac{1}{2}(\log q)xz_1+\frac{1}{8}(\log q)^2x^2z_1^2+\dotsc\right).
\end{align*}
Collecting the terms corresponding to $z_1^{-1}$, we see that
\begin{equation*}
    \text{Res}(f(z_1),0)=\frac{1}{4}A\left(\frac{1}{2};0\right)+\frac{1}{4}A\left(\frac{1}{2};0\right)x+\frac{1}{2\log q}A'\left(\frac{1}{2};0\right).
\end{equation*}
Using the fact that 
\begin{equation*}
    A\left(\frac{1}{2};0\right)=P(1)\text{ and }A'\left(\frac{1}{2};0\right)=2P'(1),
\end{equation*}
we have that 
\begin{equation*}
    \frac{1}{\pi i}\oint \frac{A\left(\frac{1}{2};z_1\right)\zeta_{\mathbb{A}}(1+2z_1)q^{\frac{x}{2}z_1}q^{-\frac{z_1}{2}}}{z_1}dz_1=\frac{1}{2}P(1)\left(x+1+\frac{4}{\log q}\frac{P'}{P}(1)\right). 
\end{equation*}
We therefore have
\begin{align*}
  \sum_{u\in\mathcal{I}_{g+1}}L\left(\frac{1}{2},\chi_u\right)&=\sum_{u\in\mathcal{I}_{g+1}}Q_1(2g+1)(1+o(1)))\\
  &=\sum_{u\in\mathcal{I}_{g+1}}P(1)\left(g+1+\frac{2}{\log q}\frac{P'}{P}(1)\right)(1+o(1))\\
  &=P(1)\left(g+1+\frac{2}{\log q}\frac{P'}{P}(1)\right)\sum_{u\in\mathcal{I}_{g+1}}1+o(q^{2g+1}).
\end{align*}
Using Lemma 2.4, we conclude that, for the first moment, Conjecture 3.1 predicts 
\begin{equation}\label{eq:5.4}
    \sum_{u\in\mathcal{I}_{g+1}}L\left(\frac{1}{2},\chi_u\right)=2\frac{P(1)}{\zeta_{\mathbb{A}}(2)}q^{2g+1}\left[g+1+\frac{2}{\log q}\frac{P'}{P}(1)\right]+o(q^{2g+1}).
\end{equation}
If we compare formulas (\ref{eq:5.4}) and (\ref{eq:2.13}), we note that the main term and principal lower order terms are the same. Hence Theorem 2.10 proves our conjecture with an error $O(g2^{\frac{g}{2}}q^{\frac{3g}{2}})$ when $k=1$. 
\subsection{Second Moment}
For the second moment, the conjecture predicts that
\begin{equation*}
    \sum_{u\in\mathcal{I}_{g+1}}L\left(\frac{1}{2},\chi_u\right)^2=Q_2(2g+1)(1+o(1)),
\end{equation*}
where 
\begin{equation*}
    Q_2(x)=-\frac{2}{(2\pi i)^2}\oint\oint\frac{G(z_1,z_2)\Delta(z_1^2,z_2^2)^2q^{\frac{x}{2}(z_1+z_2)}}{z_1^3z_2^3}dz_1dz_2.
\end{equation*}
Using MATHEMATICA, we have that
\begin{align*}
       Q_2(x)&=\frac{1}{24\log^3(q)}\Bigg[\left(x^3+6x^2+11x+6\right)A\left(\frac{1}{2};0,0\right)\log^3(q)+\left(3x^2+12x+11\right)\log^2(q)\\&\times\left(A_1\left(\frac{1}{2};0,0\right)+A_2\left(\frac{1}{2};0,0\right)\right)+12(2+x)A_{12}\left(\frac{1}{2};0,0\right)\\&-2\left(A_{222}\left(\frac{1}{2};0,0\right)-3A_{122}\left(\frac{1}{2};0,0\right)-3A_{112}\left(\frac{1}{2};0,0\right)+A_{111}\left(\frac{1}{2};0,0\right)\right)\Bigg],
\end{align*}
where $A_j$ denotes the partial derivative, evaluated at zero of the function $A\left(\frac{1}{2};z_1.\dotsc,z_k\right)$ with respect to the $j^{\text{th}}$ variable. Hence the leading order asymptotic for the second moment for this family of L-functions can be written conjectually as
\begin{equation*}
    \sum_{u\in\mathcal{I}_{g+1}}L\left(\frac{1}{2},\chi_u\right)^2\sim \frac{2}{3}\frac{q^{2g+1}}{\zeta_{\mathbb{A}}(2)}g^3A\left(\frac{1}{2};0,0\right),
\end{equation*}
when $g\rightarrow\infty$, where 
\begin{equation*}
    A\left(\frac{1}{2};0,0\right)=\prod_P\left(1-\frac{4|P|^2-3|P|+1}{|P|^3(|P|+1)}\right).
\end{equation*}
\subsection{Third Moment}
For the third moment, the conjecture predicts that
\begin{equation*}
    \sum_{u\in\mathcal{I}_{g+1}}L\left(\frac{1}{2},\chi_u\right)^3=\sum_{u\in\mathcal{I}_{g+1}}Q_3(2g+1)(1+o(1)),
\end{equation*}
where
\begin{equation*}
    Q_3(x)=-\frac{4}{3}\frac{1}{(2\pi i)^3}\oint\oint\oint\frac{G(z_1,z_2,z_3)\Delta(z_1^2,z_2^2,z_3^2)^2q^{\frac{x}{2}(z_1+z_2+z_3)}}{z_1^5z_2^5z_3^5}dz_1dz_2dz_3.
\end{equation*}
Using MATHEMATICA, we have that
 \begin{align*}
        Q_3(x)&=\frac{1}{8640\log^6(q)}\Bigg[3(3+x)^2(x^4+12x^3+49x^2+78x+40)A\left(\frac{1}{2};0,0,0\right)\log^6(q)\\
        &+4(3x^5+45x^4+260x^3+720x^2+949x+471)\log^5(q)\\
        &\times\left(A_1\left(\frac{1}{2};0,0,0\right)+A_2\left(\frac{1}{2};0,0,0\right)+A_3\left(\frac{1}{2};0,0,0\right)\right)+4(949+1440x+780x^2+180x^3+15x^4)\\
        &\times\log^4(q)\left(A_{23}\left(\frac{1}{2};0,0,0\right)+A_{13}\left(\frac{1}{2};0,0,0\right)+A_{12}\left(\frac{1}{2};0,0,0\right)\right)+10(24+26x+9x^2+x^3)\\
        &\times\log^3(q)\Bigg(2A_{333}\left(\frac{1}{2};0,0,0\right)-3A_{233}\left(\frac{1}{2};0,0,0\right)-3A_{223}\left(\frac{1}{2};0,0,0\right)+2A_{222}\left(\frac{1}{2};0,0,0\right)\\
        &-3A_{133}\left(\frac{1}{2};0,0,0\right)-36A_{123}\left(\frac{1}{2};0,0,0\right)-3A_{122}\left(\frac{1}{2};0,0,0\right)-3A_{113}\left(\frac{1}{2};0,0,0\right)-3A_{112}\left(\frac{1}{2};0,0,0\right)\\
        &+2A_{111}\left(\frac{1}{2};0,0,0\right)\Bigg)-20(26+18x+3x^2)\log^2(q)\Bigg(A_{2333}\left(\frac{1}{2};0,0,0\right)+A_{2223}\left(\frac{1}{2};0,0,0\right)\\
        &+A_{1333}\left(\frac{1}{2};0,0,0\right)-6A_{1233}\left(\frac{1}{2};0,0,0\right)-6A_{1223}\left(\frac{1}{2};0,0,0\right)+A_{1222}\left(\frac{1}{2};0,0,0\right)-6A_{1123}\left(\frac{1}{2};0,0,0\right)\\
        &+A_{1113}\left(\frac{1}{2};0,0,0\right)+A_{1112}\left(\frac{1}{2};0,0,0\right)\Bigg)+6(3+x)\log(q)\Bigg(2A_{33333}\left(\frac{1}{2};0,0,0\right)-5A_{23333}\left(\frac{1}{2};0,0,0\right)\\
        &-10A_{22333}\left(\frac{1}{2};0,0,0\right)-10A_{22233}\left(\frac{1}{2};0,0,0\right)-5A_{22223}\left(\frac{1}{2};0,0,0\right)+2A_{22222}\left(\frac{1}{2};0,0,0\right)\\
        &+5A_{13333}\left(\frac{1}{2};0,0,0\right)+60A_{12233}\left(\frac{1}{2};0,0,0\right)+5A_{12222}\left(\frac{1}{2};0,0,0\right)-10A_{11333}\left(\frac{1}{2};0,0,0\right)\\
        &+60A_{11233}\left(\frac{1}{2};0,0,0\right)+60A_{11223}\left(\frac{1}{2};0,0,0\right)-10A_{11222}\left(\frac{1}{2};0,0,0\right)-10A_{11133}\left(\frac{1}{2};0,0,0\right)\\
        &-10A_{11122}\left(\frac{1}{2};0,0,0\right)-5A_{11113}\left(\frac{1}{2};0,0,0\right)-A_{11112}\left(\frac{1}{2};0,0,0\right)+2A_{11111}\left(\frac{1}{2};0,0,0\right)\Bigg)\\
        &+4\Bigg(3A_{233333}\left(\frac{1}{2};0,0,0\right)-20A_{222333}\left(\frac{1}{2};0,0,0\right)+3A_{222223}\left(\frac{1}{2};0,0,0\right)+3A_{133333}\left(\frac{1}{2};0,0,0\right)\\
        &-30A_{123333}\left(\frac{1}{2};0,0,0\right)+30A_{122333}\left(\frac{1}{2};0,0,0\right)+30A_{122233}\left(\frac{1}{2};0,0,0\right)-30A_{122223}\left(\frac{1}{2};0,0,0\right)\\
        &+3A_{122222}\left(\frac{1}{2};0,0,0\right)+30A_{112333}\left(\frac{1}{2};0,0,0\right)+30A_{112223}\left(\frac{1}{2};0,0,0\right)-20A_{111333}\left(\frac{1}{2};0,0,0\right)\\
        &+30A_{111233}\left(\frac{1}{2};0,0,0\right)+30A_{111223}\left(\frac{1}{2};0,0,0\right)-20A_{111222}\left(\frac{1}{2};0,0,0\right)-30A_{111123}\left(\frac{1}{2};0,0,0\right)\\
        &+3A_{111113}\left(\frac{1}{2};0,0,0\right)+3A_{111112}\left(\frac{1}{2};0,0,0\right)\Bigg)\Bigg].
    \end{align*}
    Hence the leading order asymptotic for the third moment for this family of L-functions can be written conjectually as 
    \begin{equation*}
        \sum_{u\in\mathcal{I}_{g+1}}L\left(\frac{1}{2},\chi_u\right)^3\sim \frac{2}{45}\frac{q^{2g+1}}{\zeta_{\mathbb{A}}(2)}g^6A\left(\frac{1}{2};0,0,0\right),
    \end{equation*}
    when $g\rightarrow\infty$, where 
    \begin{equation*}
        A\left(\frac{1}{2};0,0,0\right)=\prod_P\left(1-\frac{12|P|^5-23|P|^4+23|P|^3-15|P|^2+6|P|-1}{|P|^6(|P|+1)}\right).
    \end{equation*}
    \section{A Conjecture for the leading order asymptotics for the moments of $L(s,\chi_u)$ for general $k$}
    In this section, we show how to obtain an explicit conjecture for the leading order asymptotic of the moments for a general integer $k$. We also use the conjecture to calculate the leading order of the asymptotic for the fourth and fifth moment. 
    \subsection{Leading order for general $k$}
    To obtain the main formula we need the following Lemma. 
    \begin{lemma}[\cite{Andrade2014}, Lemma 5]
    Let $F$ be a symmetric function in $k$ variables, regular near $(0,\dotsc,0)$ and $f(s)$ has a simple pole of residue $1$ at $s=0$ and analytic in a neighbourhood of $s=0$. Let
    \begin{equation*}
        K(q^{2g+1};w_1,\dotsc,w_k)=\sum_{\epsilon_j=\pm 1}e^{\frac{1}{2}\log(q^{2g+1})\sum_{j=1}^k\epsilon_jw_j}F(\epsilon_1w_1,\dotsc,\epsilon_kw_k)\prod_{1\leq i\leq j\leq k}f(\epsilon_iw_i+\epsilon_jw_j),
    \end{equation*}
    and define $I(q^{2g+1}, k,w=0)$ to be the value of $K$ when $w_1=\dotsc=w_k=0$. Then
    \begin{equation*}
        I(q^{2g+1},k,w=0)\sim\left(\frac{1}{2}\log\left(q^{2g+1}\right)\right)^{\frac{k(k+1)}{2}}F(0,\dotsc,0)2^{\frac{k(k+1)}{2}}\prod_{j=1}^k\frac{j!}{(2j)!}. 
    \end{equation*}
    \end{lemma}
    \par\noindent
    Recall from (\ref{eq:4.17}), we have that
    \begin{align*}
        &\sum_{u\in\mathcal{I}_{g+1}}L\left(\frac{1}{2}+\alpha_1,\chi_u\right)\dotsc L\left(\frac{1}{2}+\alpha_k,\chi_u\right)\\
        &=\left(\sum_{u\in\mathcal{I}_{g+1}}\frac{\prod_{j=1}^k(q^{2g+1})^{-\frac{1}{2}\sum_{j=1}^k\alpha_j}X\left(\frac{1}{2}+\alpha_j\right)^{\frac{1}{2}}}{(\log q)^{\frac{k(k+1)}{2}}}\sum_{\epsilon_j=\pm 1}K(\epsilon_1\alpha_1,\dotsc,\epsilon_k\alpha_k)\right)(1+o(1)),
    \end{align*}
    where
    \begin{align*}
        K(\epsilon_1\alpha_1,\dotsc,\epsilon_k\alpha_k)&=\prod_{j=1}^k(q^{2g+1})^{\frac{1}{2}\sum_{j=1}^k\epsilon_j\alpha_j}X\left(\frac{1}{2}+\epsilon_j\alpha_j\right)^{-\frac{1}{2}}\\
        &\times A\left(\frac{1}{2};\epsilon_1\alpha_1,\dotsc,\epsilon_k\alpha_k\right)\prod_{1\leq i\leq j\leq k}\zeta_{\mathbb{A}}(1+\epsilon_i\alpha_i+\epsilon_j\alpha_j)(\log q).
    \end{align*}
    Applying Lemma 6.1 with
    \begin{equation*}
        f(s)=\zeta_{\mathbb{A}}(1+s)(\log q),
    \end{equation*}
\begin{equation*}
    F(\alpha_1,\dotsc,\alpha_k)=\prod_{j=1}^kX\left(\frac{1}{2}+\alpha_j\right)^{-\frac{1}{2}}A\left(\frac{1}{2};\alpha_1,\dotsc,\alpha_k\right)
\end{equation*}
and
\begin{equation*}
    K(q^{2g+1};\alpha_1,\dotsc,\alpha_k)=\sum_{\epsilon_j=\pm 1}(q^{2g+1})^{\frac{1}{2}\sum_{j=1}^k\epsilon_j\alpha_j}F(\epsilon_1\alpha_1,\dotsc,\epsilon_k\alpha_k)\prod_{1\leq i\leq j\leq k}f(\epsilon_i\alpha_i+\epsilon_j\alpha_j),
\end{equation*}
then letting $\alpha_1,\dotsc,\alpha_k\rightarrow 0$ we obtain the following result.
\begin{thm}
Conditional on Conjecture 3.1, we have that, as $g\rightarrow \infty$, the following holds
\begin{equation*}
    \sum_{u\in\mathcal{I}_{g+1}}L\left(\frac{1}{2},\chi_u\right)^k\sim 2^{\frac{k(k+1)}{2}+1}\frac{q^{2g+1}}{\zeta_{\mathbb{A}}(2)}g^{\frac{k(k+1)}{2}}A\left(\frac{1}{2};0,\dotsc,0\right)\prod_{j=1}^k\frac{j!}{(2j)!}. 
\end{equation*}
\end{thm}
\subsection{Fourth Moment}
For the fourth moment, Theorem 6.2 gives us that
\begin{equation*}
    \sum_{u\in\mathcal{I}_{g+1}}L\left(\frac{1}{2},\chi_u\right)^4\sim\frac{2}{4725}\frac{q^{2g+1}}{\zeta_{\mathbb{A}}(2)}g^{10}A\left(\frac{1}{2};0,0,0,0\right),
\end{equation*}
where
\begin{equation*}
    A\left(\frac{1}{2};0,0,0,0\right)=\prod_P\left(1-\frac{h_4(|P|)}{|P|^{10}(|P|+1)}\right)
\end{equation*}
and
\begin{equation*}
    h_4(x)=30x^9-109x^8+210x^7-274x^6+272x^5-210x^4+119x^3-45x^2+10x-1.
\end{equation*}
\subsection{Fifth Moment}
For the fifth moment, Theorem 6.2 gives us that
\begin{equation*}
    \sum_{u\in\mathcal{I}_{g+1}}L\left(\frac{1}{2},\chi_u\right)^5\sim \frac{2}{4465125}\frac{q^{2g+1}}{\zeta_{\mathbb{A}}(2)}g^{15}A\left(\frac{1}{2};0,0,0,0,0\right),
\end{equation*}
where
\begin{equation*}
    A\left(\frac{1}{2};0,0,0,0,0\right)=\prod_P\left(1-\frac{h_5(|P|)}{|P|^{15}(|{P|+1)}}\right)
\end{equation*}
and
\begin{align*}
    h_5(x)&=65x^{14}-385x^{13}+1220x^{12}-2613x^{11}+4263x^{10}-5725x^9+6540x^8\\
    &-6275x^7+4879x^6-2965x^5+1360x^4-455x^3+105x^2-15x+1. 
\end{align*}
\section{The Ratios Conjecture for Even characteristic}
In this section, we obtain a conjectural asymptotic formula for 
\begin{equation}\label{eq:7.1}
    \sum_{u\in\mathcal{I}_{g+1}}\frac{\prod_{k=1}^KL\left(\frac{1}{2}+\alpha_k,\chi_u\right)}{\prod_{q=1}^QL\left(\frac{1}{2}+\gamma_q,\chi_u\right)},
\end{equation}
where $\mathcal{I}_{g+1}$ is defined in section 2 and the family $\mathcal{U}=\{L(s,\chi_u):u\in\mathcal{I}_{g+1}\}$ is a sympletic family. We will adapt the recipe of \cite{AndradeJungShamesaldeen2018,Andrade2014,Conrey2008a} for this family of L-functions. 
\subsection{Applying the Recipe for L-functions in Even characteristic}
First note that the L-functions in the numerator of (\ref{eq:7.1}) can be written as 
\begin{equation*}
    L(s,\chi_u)=\sum_{f\in\mathbb{A}^+_{\leq g}}\frac{\chi_u(f)}{|f|^s}+\mathcal{X}_u(s)\sum_{f\in\mathbb{A}^+_{\leq g-1}}\frac{\chi_u(f)}{|f|^{1-s}},
\end{equation*}
and those in the denominator can be written as 
\begin{equation*}
    \frac{1}{L(s,\chi_u)}=\prod_P\left(1-\frac{\chi_u(P)}{|P|^s}\right)=\sum_{f\in\mathbb{A}^+}\frac{\mu(f)\chi_u(f)}{|f|^s},
\end{equation*}
where $\mu(f)$ is defined in section 2. In the numerator, we will replace $L(s,\chi_u)$ with $Z_L(s,\chi_u)$ and so the quantity we will apply the recipe to is
\begin{align*}
    &\sum_{u\in\mathcal{I}_{g+1}}\frac{\prod_{k=1}^KZ_L\left(\frac{1}{2}+\alpha_k,\chi_u\right)}{\prod_{q=1}^QL\left(\frac{1}{2}+\gamma_q,\chi_u\right)}\\
    &=\sum_{u\in\mathcal{I}_{g+1}}\prod_{k=1}^KZ_L\left(\frac{1}{2}+\alpha_k,\chi_u\right)\sum_{\substack{h_1,\dotsc,h_Q\\h_i\text{ monic}}}\frac{\mu(h_1)\dotsc\mu(h_Q)\chi_u(h_1)\dotsc\chi_u(h_Q)}{\prod_{q=1}^Q|h_q|^{\frac{1}{2}+\gamma_q}}.
\end{align*}
From (\ref{eq:4.6}), we have that
\begin{equation*}
    \prod_{k=1}^KZ_L\left(\frac{1}{2}+\alpha_k,\chi_u\right)=\sum_{\epsilon\in\{-1,1\}^K}\prod_{k=1}^K\mathcal{X}_u\left(\frac{1}{2}+\epsilon_k\alpha_k\right)^{-\frac{1}{2}}\sum_{\substack{m_1,\dotsc,m_K\\m_j\text{ monic}}}\frac{\chi_u(m_1\dotsc m_K)}{\prod_{k=1}^K|m_j|^{\frac{1}{2}+\epsilon_k\alpha_k}},
\end{equation*}
thus
\begin{align*}
    &\sum_{u\in\mathcal{I}_{g+1}}\frac{\prod_{k=1}^KZ_L\left(\frac{1}{2}+\alpha_k,\chi_u\right)}{\prod_{q=1}^QL\left(\frac{1}{2}+\gamma_q,\chi_u\right)}\\
    &=\sum_{u\in\mathcal{I}_{g+1}}\sum_{\epsilon\in\{-1,1\}^K}\prod_{k=1}^K\mathcal{X}_u\left(\frac{1}{2}+\epsilon_k\alpha_k\right)^{-\frac{1}{2}}\sum_{\substack{m_1,\dotsc,m_K\\h_1,\dotsc,h_Q\\m_j,h_i\text{ monic}}}\frac{\prod_{q=1}^Q\mu(h_q)\chi_u\left(\prod_{k=1}^Km_k\prod_{q=1}^Qh_q\right)}{\prod_{k=1}^K|m_k|^{\frac{1}{2}+\epsilon_k\alpha_k}\prod_{q=1}^Q|h_q|^{\frac{1}{2}+\gamma_q}}.
\end{align*}
Using Lemma 4.6, we have that
\begin{align*}
    &\lim_{g\rightarrow\infty}\sum_{u\in\mathcal{I}_{g+1}}\sum_{\epsilon\in\{-1,1\}^K}\prod_{k=1}^KX\left(\frac{1}{2}+\epsilon_k\alpha_k\right)^{-\frac{1}{2}}\sum_{\substack{m_1,\dotsc,m_K\\h_1,\dotsc,h_Q\\m_j,h_i\text{ monic}}}\frac{\prod_{q=1}^Q\mu(h_q)\chi_u\left(\prod_{k=1}^Km_k\prod_{q=1}^Qh_q\right)}{\prod_{k=1}^K|m_k|^{\frac{1}{2}+\epsilon_k\alpha_k}\prod_{q=1}^Q|h_q|^{\frac{1}{2}+\gamma_q}}\\
    &=\sum_{\epsilon\in\{-1,1\}^K}\prod_{k=1}^KX\left(\frac{1}{2}+\epsilon_k\alpha_k\right)^{-\frac{1}{2}}\sum_{\substack{m_1,\dotsc,m_K\\h_1,\dotsc,h_Q\\m_j,h_i\text{ monic}}}\frac{\prod_{q=1}^Q\mu(h_q)\delta\left(\prod_{k=1}^Km_k\prod_{q=1}^Qh_q\right)}{\prod_{k=1}^K|m_k|^{\frac{1}{2}+\epsilon_k\alpha_k}\prod_{q=1}^Q|h_q|^{\frac{1}{2}+\gamma_q}},
\end{align*}
where
\[
\delta(n)=
\begin{cases}
\prod_{P|n}\left(1+\frac{1}{|P|}\right)^{-1}&\text{if }n\text{ is a square},\\
0&\text{otherwise}.
\end{cases}
\]
Let
\begin{equation}
    G_{\mathcal{U}}(\alpha;\gamma)=\sum_{\substack{m_1,\dotsc,m_K\\h_1,\dotsc,h_Q\\m_j,h_i\text{ monic}}}\frac{\prod_{q=1}^Q\mu(h_q)\delta\left(\prod_{k=1}^Km_k\prod_{q=1}^Qh_q\right)}{\prod_{k=1}^K|m_k|^{\frac{1}{2}+\alpha_k}\prod_{q=1}^Q|h_q|^{\frac{1}{2}+\gamma_q}},
\end{equation}
then we can express $G_{\mathcal{U}}(\alpha;\gamma)$ as a convergent Euler product provided that $\Re(\alpha_k)>0$ and $\Re(\gamma_q)>0$. Thus
\begin{equation*}
    G_{\mathcal{U}}(\alpha;\gamma)=\prod_P\left(1+\left(1+\frac{1}{|P|}\right)^{-1}\sum_{0<\sum_ka_k+\sum_qc_q\text{ is even}}\frac{\prod_{q=1}^Q\mu(P^{c_q})}{|P|^{\sum_ka_k\left(\frac{1}{2}+\alpha_k\right)+\sum_qc_q\left(\frac{1}{2}+\gamma_q\right)}}\right).
\end{equation*}
The above expression will enable us to locate the zeros and poles. Therefore we have
\begin{align*}
    &G_{\mathcal{U}}(\alpha;\gamma)=\prod\Bigg(1+\left(1+\frac{1}{|P|}\right)^{-1}\\
    &\times \left[\sum_{1\leq j\leq k\leq K}\frac{1}{|P|^{\left(\frac{1}{2}+\alpha_j\right)+\left(\frac{1}{2}+\alpha_k\right)}}+\sum_{1\leq q<r\leq Q}\frac{\mu(P)^2}{|P|^{\left(\frac{1}{2}+\gamma_q\right)+\left(\frac{1}{2}+\gamma_r\right)}}+\sum_{k=1}^K\sum_{q=1}^Q\frac{\mu(P)}{|P|^{\left(\frac{1}{2}+\alpha_k\right)+\left(\frac{1}{2}+\gamma_q\right)}}+\dotsc\right]\Bigg),
\end{align*}
where $\dotsc$ indicates that the terms converge. We have that the terms with $\sum_{k=1}^Ka_k+\sum_{q=1}^Qc_q=2$ contribute to the poles and zeros. The poles come from when the terms with $a_j=a_k=1$ for $1\leq j\leq k\leq K$ and from when $c_q=c_r=1$ for $1\leq q<r\leq Q$. The contribution of the zeros arrive from terms with $a_k=c_q=1$ with $1\leq k\leq K$ and $1\leq q\leq Q$. The contribution of all these zeros and poles is
\begin{equation*}
    Y_{\mathcal{U}}(\alpha;\gamma)=\frac{\prod_{1\leq j\leq k\leq K}\zeta_{\mathbb{A}}(1+\alpha_j+\alpha_k)\prod_{1\leq q<r\leq Q}\zeta_{\mathbb{A}}(1+\gamma_q+\gamma_r)}{\prod_{k=1}^K\prod_{q=1}^Q\zeta_{\mathbb{A}}(1+\alpha_k+\gamma_q)}.
\end{equation*}
So when we factor out $Y_{\mathcal{U}}$ from $G_{\mathcal{U}}$, we are left with an Euler product $A_{\mathcal{U}}$, which is absolutely convergent for all the variables in the small discs around zero:
\begin{align*}
    A_{\mathcal{U}}(\alpha;\gamma)&=\prod_P\frac{\prod_{1\leq j\leq k\leq K}\left(1-\frac{1}{|P|^{1+\alpha_j+\alpha_k}}\right)\prod_{1\leq q<r\leq Q}\left(1-\frac{1}{|P|^{1+\gamma_q+\gamma_r}}\right)}{\prod_{k=1}^K\prod_{q=1}^Q\left(1-\frac{1}{|P|^{1+\alpha_k+\gamma_q}}\right)}\\
    &\times \left(1+\left(1+\frac{1}{|P|}\right)^{-1}\sum_{0<\sum_k a_k+\sum_qc_q\text{ is even}}\frac{\prod_{q=1}^Q\mu(P^{c_q})}{|P|^{\sum_ka_k\left(\frac{1}{2}+\alpha_k\right)+\sum_qc_q\left(\frac{1}{2}+\gamma_q\right)}}\right).
\end{align*}
So we conclude that
\begin{align*}
    &\sum_{u\in\mathcal{I}_{g+1}}\frac{\prod_{k=1}^KZ_L\left(\frac{1}{2}+\alpha_k,\chi_u\right)}{\prod_{q=1}^QL\left(\frac{1}{2}+\gamma_q,\chi_u\right)}\\
    &=\sum_{u\in\mathcal{I}_{g+1}}\sum_{\epsilon\in\{-1,1\}^K}\prod_{k=1}^K\mathcal{X}_u\left(\frac{1}{2}+\epsilon_k\alpha_k\right)^{-\frac{1}{2}}A_{\mathcal{U}}(\epsilon_1\alpha_1,\dotsc,\epsilon_K\alpha_K;\gamma)Y_{\mathcal{U}}(\epsilon_1\alpha_1,\dotsc,\epsilon_K\alpha_K;\gamma)+o(q^{2g+1}).
\end{align*}
Using the definition of $Z_L(s,\chi_u)$ and $\mathcal{X}_u(s)$, we have
\begin{align}\label{eq:7.3}
    &\sum_{u\in\mathcal{I}_{g+1}}\frac{\prod_{k=1}^KL\left(\frac{1}{2}+\alpha_k,\chi_u\right)}{\prod_{q=1}^QL\left(\frac{1}{2}+\gamma_q,\chi_u\right)}\nonumber\\
    &=\sum_{u\in\mathcal{I}_{g+1}}\sum_{\epsilon\in\{-1,1\}^K}(q^{2g+1})^{\frac{1}{2}\sum_{k=1}^K(\epsilon_k\alpha_k-\alpha_k)}\prod_{k=1}^KX\left(\frac{1}{2}+\alpha_k\right)^{\frac{1}{2}}X\left(\frac{1}{2}+\epsilon_k\alpha_k\right)^{-\frac{1}{2}}\nonumber\\
    &\times A_{\mathcal{U}}(\epsilon_1\alpha_1,\dotsc,\epsilon_K\alpha_K;\gamma)Y_{\mathcal{U}}(\epsilon_1\alpha_1,\dotsc,\epsilon_K\alpha_K;\gamma)+o(q^{2g+1}). 
\end{align}
\begin{lemma}
We have
\begin{equation*}
    X\left(\frac{1}{2}+\alpha_k\right)^{\frac{1}{2}}X\left(\frac{1}{2}+\epsilon_k\alpha_k\right)^{-\frac{1}{2}}=X\left(\frac{1}{2}+\frac{\alpha_k-\epsilon_k\alpha_k}{2}\right).
\end{equation*}
\end{lemma}
\begin{proof}
The proof follows directly from the definition of $X(s)$.
\end{proof}
\par\noindent
Therefore, using Lemma 7.1, (\ref{eq:7.3}) gives us
\begin{align*}
    &\sum_{u\in\mathcal{I}_{g+1}}\frac{\prod_{k=1}^KL\left(\frac{1}{2}+\alpha_k,\chi_u\right)}{\prod_{q=1}^QL\left(\frac{1}{2}+\gamma_q,\chi_u\right)}\\
    &=\sum_{u\in\mathcal{I}_{g+1}}\sum_{\epsilon\in\{-1,1\}^K}(q^{2g+1})^{\frac{1}{2}\sum_{k=1}^K(\epsilon_k\alpha_k-\alpha_k)}\prod_{k=1}^KX\left(\frac{1}{2}+\frac{\alpha_k-\epsilon_k\alpha_k}{2}\right)\\
    &\times A_{\mathcal{U}}(\epsilon_1\alpha_1,\dotsc,\epsilon_K\alpha_K;\gamma)Y_{\mathcal{U}}(\epsilon_1\alpha_1,\dotsc,\epsilon_K\alpha_K;\gamma)+o(q^{2g+1}).
\end{align*}
If we let 
\begin{equation*}
    H_{\mathcal{I},\alpha,\gamma}(w)=(q^{2g+1})^{\frac{1}{2}\sum_{k=1}^Kw_k}\prod_{k=1}^KX\left(\frac{1}{2}+\frac{\alpha_k-w_k}{2}\right)A_{\mathcal{U}}(w_1,\dotsc,w_K;\gamma)Y_{\mathcal{U}}(w_1,\dotsc,w_k;\gamma),
\end{equation*}
then 
\begin{equation*}
    \sum_{u\in\mathcal{I}_{g+1}}\frac{\prod_{k=1}^KL\left(\frac{1}{2}+\alpha_k,\chi_u\right)}{\prod_{q=1}^QL\left(\frac{1}{2}+\gamma_q,\chi_u\right)}=\sum_{u\in\mathcal{I}_{g+1}}(q^{2g+1})^{-\frac{1}{2}\sum_{k=1}^K\alpha_k}\sum_{\epsilon\in\{-1,1\}^K}H_{\mathcal{I},\alpha,\gamma}(\epsilon_1\alpha_1,\dotsc,\epsilon_K\alpha_K)+o(q^{2g+1}),
\end{equation*}
which is precisely the formulae given in Conjecture 3.2.
\subsection{The Closed form Expression for $A_{\mathcal{U}}$}
In this subsection, we refine the conjecture by deriving a closed form expression for the Euler product $A_{\mathcal{U}}(\alpha;\gamma)$. Suppose that
\begin{equation*}
    f(x)=\sum_{n=1}^{\infty}a_nu^n,
\end{equation*}
then 
\begin{equation*}
    \sum_{o<n\text{ is even}}u_nx^n=\frac{1}{2}\left(f(x)+f(-x)-2\right).
\end{equation*}
Thus
\begin{align}\label{eq:7.4}
    1+\left(1+\frac{1}{|P|}\right)^{-1}\sum_{0<n\text{ is even}}u_nx^n&=1+\left(1+\frac{1}{|P|}\right)^{-1}\left(\frac{1}{2}\left(f(x)+f(-x)-2\right)\right)\nonumber\\
    &=\frac{1}{1+\frac{1}{|P|}}\left(\frac{f(x)+f(-x)}{2}+\frac{1}{|P|}\right). 
\end{align}
Now let 
\begin{align}\label{eq:7.5}
    f\left(\frac{1}{|P|}\right)&=\sum_{a_k,c_q}\frac{\prod_{q=1}^Q\mu(P^{c_q})}{|P|^{\sum_ka_k\left(\frac{1}{2}+\alpha_k\right)+\sum_qc_q\left(\frac{1}{2}+\gamma_q\right)}}\nonumber\\
    &=\sum_{a_k}\frac{1}{|P|^{\sum_ka_k\left(\frac{1}{2}+\alpha_k\right)}}\sum_{c_q}\frac{\prod_{q=1}^Q\mu(P^{c_q})}{|P|^{\sum_qc_q\left(\frac{1}{2}+\gamma_q\right)}}\nonumber\\
    &=\sum_{a_k}\prod_{k=1}^K\frac{1}{|P|^{a_k\left(\frac{1}{2}+\alpha_k\right)}}\sum_{c_q}\prod_{q=1}^Q\frac{\mu(P^{c_q})}{|P|^{c_q\left(\frac{1}{2}+\gamma_q\right)}}\nonumber\\
    &=\frac{\prod_{q=1}^Q\left(1-\frac{1}{|P|^{\frac{1}{2}+\gamma_q}}\right)}{\prod_{k=1}^K\left(1-\frac{1}{|P|^{\frac{1}{2}+\alpha_k}}\right)},
\end{align}
which leads us to the following Lemma. 
\begin{lemma}
We have
\begin{align*}
    &1+\left(1+\frac{1}{|P|}\right)^{-1}\sum_{o<\sum_ka_k+\sum_qc_q\text{ is even}}\frac{\prod_{q=1}^Q\mu(P^{c_q})}{|P|^{\sum_ka_k\left(\frac{1}{2}+\alpha_k\right)+\sum_qc_q\left(\frac{1}{2}+\gamma_q\right)}}\\
    &=\frac{1}{1+\frac{1}{|P|}}\left(\frac{1}{2}\frac{\prod_{q=1}^Q\left(1-\frac{1}{|P|^{\frac{1}{2}+\gamma_q}}\right)}{\prod_{k=1}^K\left(1-\frac{1}{|P|^{\frac{1}{2}+\alpha_k}}\right)}+\frac{1}{2}\frac{\prod_{q=1}^Q\left(1+\frac{1}{|P|^{\frac{1}{2}+\gamma_q}}\right)}{\prod_{k=1}^K\left(1+\frac{1}{|P|^{\frac{1}{2}+\alpha_k}}\right)}+\frac{1}{|P|}\right).
\end{align*}
\end{lemma}
\begin{proof}
The proof follows from (\ref{eq:7.4}) and (\ref{eq:7.5}). 
\end{proof}
\par\noindent
The next Corollary is immediate from Lemma 7.2.
\begin{cor}
We have
\begin{align*}
    A_{\mathcal{U}}(\alpha;\gamma)&=\prod_P\frac{\prod_{1\leq j\leq k\leq K}\left(1-\frac{1}{|P|^{1+\alpha_j+\alpha_k}}\right)\prod_{1\leq q<r\leq Q}\left(1-\frac{1}{|P|^{1+\gamma_q+\gamma_r}}\right)}{\prod_{k=1}^K\prod_{q=1}^Q\left(1-\frac{1}{|P|^{1+\alpha_k+\gamma_q}}\right)}\\
    &\times \frac{1}{1+\frac{1}{|P|}}\left(\frac{1}{2}\frac{\prod_{q=1}^Q\left(1-\frac{1}{|P|^{\frac{1}{2}+\gamma_q}}\right)}{\prod_{k=1}^K\left(1-\frac{1}{|P|^{\frac{1}{2}+\alpha_k}}\right)}+\frac{1}{2}\frac{\prod_{q=1}^Q\left(1+\frac{1}{|P|^{\frac{1}{2}+\gamma_q}}\right)}{\prod_{k=1}^K\left(1+\frac{1}{|P|^{\frac{1}{2}+\alpha_k}}\right)}+\frac{1}{|P|}\right).
\end{align*}
\end{cor}
\subsection{The Final form of the Conjecture}
In this subsection, we express the combinatorial sum as a multiple integral. We begin by stating the following Lemma from \cite{Conrey2008a}. 
\begin{lemma}
Suppose that $F(z)=F(z_1,\dotsc,z_K)$ is a function in $K$ variables, which is symmetric and regular near $(0,\dotsc,0)$. Suppose further that $f(s)$ has a simple pole of residue 1 at $s=0$ but is otherwise analytic in $|s|\leq 1$. Let either
\begin{equation*}
    H(z_1,\dotsc,z_K)=F(z_1,\dotsc,z_K)\prod_{1\leq j\leq k\leq K}f(z_j+z_k)
\end{equation*}
or 
\begin{equation*}
    H(z_1,\dotsc,z_K)=F(z_1,\dotsc,z_K)\prod_{1\leq j<k\leq K}f(z_j+z_k).
\end{equation*}
If $|\alpha_k|<1$, then 
\begin{align*}
    &\sum_{\epsilon\in\{-1,1\}^K}H(\epsilon_1\alpha_1,\dotsc,\epsilon_K\alpha_K)\\
    &=\frac{(-1)^{\frac{K(K-1)}{2}}2^K}{K!}\frac{1}{(2\pi i)^K}\int_{|z_i|=1}\frac{H(z_1,\dotsc,z_K)\Delta(z_1^2,\dotsc,z_K^2)^2\prod_{k=1}^Kz_k}{\prod_{j=1}^K\prod_{k=1}^K(z_k-\alpha_j)(z_k+\alpha_j)}dz_1\dotsc dz_K,
\end{align*}
and
\begin{align*}
    &\sum_{\epsilon\in\{-1,1\}^K}sgn(\epsilon)H(\epsilon_1\alpha_1,\dotsc,\epsilon_K\alpha_K)\\
    &=\frac{(-1)^{\frac{K(K-1)}{2}}2^K}{K!}\frac{1}{(2\pi i)^K}\int_{|z_i|=1}\frac{H(z_1,\dotsc,z_K)\Delta(z_1^2,\dotsc,z_K^2)\prod_{k=1}^K\alpha_k}{\prod_{j=1}^K\prod_{k=1}^K(z_k-\alpha_j)(z_k+\alpha_j)}dz_1\dotsc dz_K.
\end{align*}
\end{lemma}
\par\noindent
Now we are in a position to present the final form of the ratios conjecture for L-functions in even characteristic using the integrals introduced above. Therefore, Conjecture 3.2 can be written as follows.
\begin{conj}
Suppose that the real parts of $\alpha_k$ and $\gamma_q$ are positive. Then
\begin{align*}
    &\sum_{u\in\mathcal{I}_{g+1}}\frac{\prod_{k=1}^KL\left(\frac{1}{2}+\alpha_k,\chi_u\right)}{\prod_{q=1}^QL\left(\frac{1}{2}+\gamma_q,\chi_u\right)}\\
    &=\sum_{u\in\mathcal{I}_{g+1}}(q^{2g+1})^{-\frac{1}{2}\sum_{k=1}^K\alpha_k}\frac{(-1)^{\frac{K(K-1)}{2}}2^K}{K!}\int_{|z_i|=1}\frac{H_{\mathcal{I},\alpha,\gamma}(z_1,\dotsc,z_K)\Delta(z_1^2,\dotsc,z_K^2)^2\prod_{k=1}^Kz_k}{\prod_{j=1}^K\prod_{k=1}^K(z_k-\alpha_j)(z_k+\alpha_j)}dz_1\dotsc dz_K.
\end{align*}
\end{conj}
\section{Applications of the Ratios Conjecture}
In this section we present an application of the Ratios Conjecture 3.2. Specifically we derive a formula for the one-level density. 
\subsection{One-level Density}
Our goal is to consider
\begin{equation*}
    R_{\mathcal{U}}(\alpha;\gamma)=\sum_{u\in\mathcal{I}_{g+1}}\frac{L\left(\frac{1}{2}+\alpha,\chi_u\right)}{L\left(\frac{1}{2}+\gamma,\chi_u\right)}.
\end{equation*}
In this case, the conjecture is.
\begin{conj}
With $-\frac{1}{4}<\Re(\alpha)<\frac{1}{4}, \frac{1}{\log(q^{2g+1})}\ll \Re(\gamma)<\frac{1}{4}$ and $\Im(\alpha), \Im(\gamma)\ll (q^{2g+1})^{1-\epsilon}$ for every $\epsilon>0$, we have
\begin{align}
    R_{\mathcal{U}}(\alpha;\gamma)&=\sum_{u\in\mathcal{I}_{g+1}}\frac{L\left(\frac{1}{2}+\alpha,\chi_u\right)}{L\left(\frac{1}{2}+\gamma,\chi_u\right)}\nonumber\\
    &=\sum_{u\in\mathcal{I}_{g+1}}\left(A_{\mathcal{U}}(\alpha;\gamma)\frac{\zeta_{\mathbb{A}}(1+2\alpha)}{\zeta_{\mathbb{A}}(1+\alpha+\gamma)}+(q^{2g+1})^{-\alpha}X\left(\frac{1}{2}+\alpha\right)A_{\mathcal{U}}(-\alpha;\gamma)\frac{\zeta_{\mathbb{A}}(1-2\alpha)}{\zeta_{\mathbb{A}}(1-\alpha+\gamma)}\right)\nonumber\\
    &+o(q^{2g+1}),
\end{align}
where
\begin{equation*}
    A_{\mathcal{U}}(\alpha;\gamma)=\prod_P\left(1-\frac{1}{|P|^{1+\alpha+\gamma}}\right)^{-1}\left(1-\frac{1}{|P|^{1+2\alpha}(|P|+1)}-\frac{1}{|P|^{\alpha+\gamma}(|P|+1)}\right).
\end{equation*}
\end{conj}
\par\noindent
To obtain the one-level density, we note that
\begin{equation*}
    \sum_{u\in\mathcal{I}_{g+1}}\frac{L'\left(\frac{1}{2}+r,\chi_u\right)}{L\left(\frac{1}{2}+r,\chi_u\right)}=\frac{\partial}{\partial \alpha}R_{\mathcal{U}}(\alpha;\gamma)\Bigg|_{\alpha=\gamma=r}.
\end{equation*}
A direct calculation gives us that
\begin{equation*}
    \frac{\partial}{\partial \alpha}\frac{\zeta_{\mathbb{A}}(1+2\alpha)}{\zeta_{\mathbb{A}}(1+\alpha+\gamma)}\Bigg|_{\alpha=\gamma=r}=\frac{\zeta_{\mathbb{A}}'(1+2r)}{\zeta_{\mathbb{A}}(1+2r)}A_{\mathcal{U}}(r;r)+A_{\mathcal{U}}'(r;r)
\end{equation*}
and
\begin{align*}
    &\frac{\partial}{\partial\alpha}\left((q^{2g+1})^{-\alpha}X\left(\frac{1}{2}+\alpha\right)\frac{\zeta_{\mathbb{A}}(1-2\alpha)}{\zeta_{\mathbb{A}}(1-\alpha+\gamma)}A_{\mathcal{U}}(-\alpha;\gamma)\right)\Bigg|_{\alpha=\gamma=r}\\
    &=-(\log q)(q^{2g+1})^{-r}X\left(\frac{1}{2}+r\right)\zeta_{\mathbb{A}}(1-2r)A_{\mathcal{U}}(-r;r),
\end{align*}
where
\begin{equation*}
    A_{\mathcal{U}}(r;r)=1,
\end{equation*}
\begin{equation*}
    A_{\mathcal{U}}(-r;r)=\prod_P\left(1-\frac{1}{|P|}\right)^{-1}\left(1-\frac{1}{(|P|+1)|P|^{1-2r}}-\frac{1}{|P|+1}\right)
\end{equation*}
and
\begin{equation*}
    A_{\mathcal{U}}'(r,r)=\sum_P\frac{\log|P|}{(|P|^{1+2r}-1)(|P|+1)}.
\end{equation*}
Therefore the ratios conjecture implies the following holds.
\begin{thm}
Assuming Conjecture 8.1, $\frac{1}{\log(q^{2g+1})}\ll\Re(r)<\frac{1}{4}$ and $\Im(r)\ll_{\epsilon}(q^{2g+1})^{1-\epsilon}$, we have
\begin{align*}
    &\sum_{u\in\mathcal{I}_{g+1}}\frac{L'\left(\frac{1}{2}+r,\chi_u\right)}{L\left(\frac{1}{2}+r,\chi_u\right)}\\
    &=\sum_{u\in\mathcal{I}_{g+1}}\left(\frac{\zeta_{\mathbb{A}}'(1+2r)}{\zeta_{\mathbb{A}}(1+2r)}A_{\mathcal{U}}'(r;r)-(\log q)(q^{2g+1})^{-r}X\left(\frac{1}{2}+r\right)\zeta_{\mathbb{A}}(1-2r)A_{\mathcal{U}}(-r;r)\right)+o(q^{2g+1}).
\end{align*}
\end{thm}
\par\noindent
Now we are in a position to derive a formula for the one-level density for the zeros of quadratic Dirichlet L-functions in even characteristic, complete with lower order terms. Let $\chi_u$ denote the ordinate of a generic zero of $L(s,\chi_u)$ on the half-line. As $L(s,\chi_u)$ is a function in $q^{-s}$ and so is periodic with period $\frac{2\pi i}{\log q}$, we can confine our analysis of the zeros to $-\frac{\pi i}{\log q}\leq \Im(s)\leq \frac{\pi i}{\log q}$. So we consider the one-level density
\begin{equation*}
    S_1(f)=\sum_{u\in\mathcal{I}_{g+1}}\sum_{\gamma_u}f(\gamma_u),
\end{equation*}
where $f$ is a $\frac{2\pi}{\log q}$ periodic even test function and holomorphic. By Cauchy's argument principal, we have that
\begin{equation*}
    S_1(f)=\sum_{u\in\mathcal{I}_{g+1}}\left(\int_{(c)}-\int_{(1-c)}\right)\frac{L'(s,\chi_u)}{L(s,\chi_u)}f\left(-i\left(s-\frac{1}{2}\right)\right)ds,
\end{equation*}
where $(c)$ denotes a vertical line from $c-\frac{\pi i}{\log q}$ to $c+\frac{\pi i}{\log q}$ and $\frac{1}{2}+\frac{1}{\log(q^{2g+1})}<c<\frac{3}{4}$. The integral along the $(c)$-line is equal to
\begin{equation*}
    \frac{1}{2\pi}\int_{-\frac{\pi}{\log q}}^{\frac{\pi}{\log q}}f\left(-i\left(c+it-\frac{1}{2}\right)\right)\sum_{u\in\mathcal{I}_{g+1}}\frac{L'(c+it,\chi_u)}{L(c+it,\chi_u)}dt.
\end{equation*}
Moving the path of integration to $c=\frac{1}{2}$ as the integral is regular at $t=0$ and using Theorem 8.2, we get that the integral along the $(c)$-line is equal to
\begin{align}\label{eq:8.2}
    \frac{1}{2\pi}\int_{-\frac{\pi }{\log q}}^{\frac{\pi}{\log q}}f(t)\sum_{u\in\mathcal{I}_{g+1}}&\Bigg(\frac{\zeta_{\mathbb{A}}'(1+2it)}{\zeta_{\mathbb{A}}(1+2it)}+A_{\mathcal{U}}'(it;it)\nonumber\\
    &-(\log q)(q^{2g+1})^{-it}X\left(\frac{1}{2}+it\right)\zeta_{\mathbb{A}}(1-2it)A_{\mathcal{U}}(-it;it)\Bigg)dt.
    \end{align}
    For the integral along the $(1-c)$-line, we use the change of substitution $s\rightarrow 1-s$ and we use the functional equation
    \begin{equation*}
        \frac{L'(1-s,\chi_u)}{L(1-s,\chi_u)}=\frac{\mathcal{X}'_u(s)}{\mathcal{X}_u(s)}-\frac{L'(s,\chi_u)}{L(s,\chi_u)},
    \end{equation*}
    where
    \begin{equation*}
        \frac{\mathcal{X}'_u(s)}{\mathcal{X}_u(s)}=-\log(q^{2g+1})+\frac{X'(s)}{X(s)}.
    \end{equation*}
    Thus the integration along the $(1-c)$-line is equal to 
    \begin{align}\label{eq:8.3}
        &\frac{1}{2\pi}\int_{-\frac{\pi}{\log q}}^{\frac{\pi}{\log q}}f(t)\sum_{u\in\mathcal{I}_{g+1}}\Bigg(\log(q^{2g+1})+\frac{X'}{X}\left(\frac{1}{2}-it\right)\nonumber\\
        &+\left(\frac{\zeta_{\mathbb{A}}'(1+2it)}{\zeta_{\mathbb{A}}(1+2it)}+A_{\mathcal{U}}'(it;it)-(\log q)(q^{2g+1})^{-it}X\left(\frac{1}{2}+it\right)\zeta_{\mathbb{A}}(1-2it)A_{\mathcal{U}}(-it,it)\right)\Bigg)+o(q^{2g+1}). 
    \end{align}
    Combining (\ref{eq:8.2}) and (\ref{eq:8.3}), we obtain the following Theorem.
    \begin{thm}
    Assuming the ratios conjecture 8.1, the one-level density of quadratic Dirichlet L-functions in even characteristic is given by
    \begin{align}
        S_1(f)&=\sum_{u\in\mathcal{I}_{g+1}}\sum_{\gamma_u}f(\gamma_u)\nonumber\\
        &=\frac{1}{2\pi}\int_{-\frac{\pi}{\log q}}^{\frac{\pi}{\log q}}f(t)\sum_{u\in\mathcal{I}_{g+1}}\Bigg(\log(q^{2g+1})+\frac{X'}{X}\left(\frac{1}{2}-it\right)\nonumber\\
        &+2\left(\frac{\zeta_{\mathbb{A}}'(1+2it)}{\zeta_{\mathbb{A}}(1+2it)}+A_{\mathcal{U}}'(it;it)-(\log q)(q^{2g+1})^{-it}X\left(\frac{1}{2}+it\right)\zeta_{\mathbb{A}}(1-2it)A_{\mathcal{U}}(-it;it)\right)\Bigg)+o(q^{2g+1}),
    \end{align}
    where $\gamma_u$ is a generic zero of $L(s,\chi_u)$ and $f$ is a periodic and nice test function. 
    \end{thm}
\subsection{Scaled One-level density}
Defining 
\begin{equation*}
    f(t)=h\left(\frac{t(2g\log q)}{2\pi}\right)
\end{equation*}
and scale the variable $t$ form Theorem 8.3 as 
\begin{equation*}
    \tau=\frac{t(2g\log q)}{2\pi}.
\end{equation*}
We have that
\begin{align}\label{eq:8.5}
    &\sum_{u\in\mathcal{I}_{g+1}}\sum_{\gamma_u}f\left(\gamma_u\frac{2g\log q}{2\pi}\right)\nonumber\\
    &=\frac{1}{2g\log q}\int_{-g}^gh(\tau)\sum_{u\in\mathcal{I}_{g+1}}\Bigg(\log(q^{2g+1})+\frac{X'}{X}\left(\frac{1}{2}-\frac{2\pi i\tau}{2g\log q}\right)\nonumber\\
    &+2\Bigg(\frac{\zeta_{\mathbb{A}}'\left(1+\frac{4\pi i\tau}{2g\log q}\right)}{\zeta_{\mathbb{A}}\left(1+\frac{4\pi i\tau}{2g\log q}\right)}+A'_{\mathcal{U}}\left(\frac{2\pi i\tau}{2g\log q};\frac{2\pi i\tau}{2g\log q}\right)-(\log q)e^{-\frac{2\pi i\tau}{2g\log q}\log(q^{2g+1})}\nonumber\\
    &\times X\left(\frac{1}{2}+\frac{2\pi i\tau}{2g\log q}\right)\zeta_{\mathbb{A}}\left(1-\frac{4\pi i\tau}{2g\log q}\right)A_{\mathcal{U}}\left(-\frac{2\pi i\tau}{2g\log q};\frac{2\pi i\tau}{2g\log q}\right)\Bigg)\Bigg)d\tau+o(q^{2g+1}).
\end{align}
Writing
\begin{equation*}
    \zeta_{\mathbb{A}}(1+s)=\frac{1}{\log q}\frac{1}{s}+\frac{1}{2}+\frac{1}{12}(\log q)s+O(s^2),
\end{equation*}
we have
\begin{equation*}
    \frac{\zeta_{\mathbb{A}}'(1+s)}{\zeta_{\mathbb{A}}(1+s)}=-\frac{1}{s}+\frac{1}{2}\log q-\frac{1}{12}(\log q)^2s+O(s^3).
\end{equation*}
Therefore as $g\rightarrow\infty$, only $\log(q^{2g+1})$ term, the $\frac{\zeta'_{\mathbb{A}}}{\zeta_{\mathbb{A}}}$ term and the final term in the integral (\ref{eq:8.5}) contribute, yielding the asymptotic 
\begin{align*}
    &\sum_{u\in\mathcal{I}_{g+1}}\sum_{\gamma_u}f\left(\gamma_u\frac{2g\log q}{2\pi}\right)\\
    &\sim \frac{1}{2g\log q}\int_{-\infty}^{\infty}h(\tau)\left((\#\mathcal{I}_{g+1})\log(q^{2g+1})-(\#\mathcal{I}_{g+1})\frac{2g\log q}{2\pi i\tau}+(\#\mathcal{I}_{g+1})\frac{2g\log q}{2\pi i\tau}e^{-2\pi i \tau}\right)d\tau.
\end{align*}
Since $h$ is an even function, the middle term drops out and the last term can be duplicated with a change of sign, leaving
\begin{align*}
    &\lim_{g\rightarrow\infty}\frac{1}{\#\mathcal{I}_{g+1}}\sum_{u\in\mathcal{I}_{g+1}}\sum_{\gamma_u}f\left(\gamma_u\frac{2g\log q}{2\pi}\right)\\
    &=\int_{-\infty}^{\infty}h(\tau)\left(1+\frac{e^{-2\pi i\tau}}{4\pi i\tau}-\frac{e^{2\pi i \tau}}{4\pi i\tau}\right)d\tau\\
    &=\int_{-\infty}^{\infty}h(\tau)\left(1+\frac{1}{4\pi i\tau}\left(-(\cos(2\pi \tau)+i\sin(2\pi \tau))+(\cos(2\pi \tau)-i\sin(2\pi \tau))\right)\right)d\tau\\
    &=\int_{-\infty}^{\infty}h(\tau)\left(1-\frac{\sin(2\pi\tau)}{2\pi \tau}\right)d\tau.
    \end{align*}
    \textbf{Acknowledgement:} The author is grateful to the Leverhulme Trust (RPG-2017-320) for the support given during this research through a PhD studentship. The author would also like to thank Dr. Julio Andrade for suggesting this problem to me, as well as his useful advice during the course of the research.
\bibliographystyle{plain}
\bibliography{evencharacteristic}

\begin{thebibliography}{10}

\bibitem{Andrade2016}
J.C. Andrade, S.~Bae, and H.~Jung.
\newblock {Average values of L-series for real characters in function fields}.
\newblock {\em Res. Math. Sci.}, 3(1):1--47, 2016.

\bibitem{AndradeJungShamesaldeen2018}
J.C. Andrade, H.~Jung, and A.~Shamesaldeen.
\newblock {The integral Moments and Ratios of Quadratic Dirichlet L-functions
  over monic irreducible Polynomials in $F_q[T]$}.
\newblock {\em Ramanujan J.}, 56:23--66, 2021.

\bibitem{Andrade2012}
J.C. Andrade and J.P. Keating.
\newblock {The mean value of L(1/2, $\chi$) in the hyperelliptic ensemble}.
\newblock {\em J. Number Theory}, 132:2793--2816, 2012.

\bibitem{Andrade2013}
J.C. Andrade and J.P. Keating.
\newblock {Mean value theorems for L-functions over prime polynomials for the
  rational function field}.
\newblock {\em Acta Arith.}, 161(4):371--385, 2013.

\bibitem{Andrade2014}
J.C. Andrade and J.P. Keating.
\newblock {Conjectures for the integral moments and ratios of L-functions over
  function fields}.
\newblock {\em J. Number Theory}, 142:102--148, 2014.

\bibitem{Bae2018}
S.~Bae and H.~Jung.
\newblock {Average values of L-functions in even characteristic}.
\newblock {\em J. Number Theory}, 186:269--303, 2018.

\bibitem{Bui2018}
H.~Bui and A.~Florea.
\newblock {Zeros of quadratic Dirichlet L-functions in the hyperelliptic
  ensemble}.
\newblock {\em Trans. Amer. Math. Soc}, 370(11):8013--8045, 2018.

\bibitem{Bui2020}
H.~Bui and A.~Florea.
\newblock {Moments of Dirichlet L–functions with prime conductors over
  function fields}.
\newblock {\em Finite Fields Appl.}, 64:1--21, 2020.

\bibitem{Bui2021}
H.~Bui, A.~Florea, and J.P. Keating.
\newblock {Type-I contributions to the one and two level densities of quadratic
  Dirichlet L–functions over function fields}.
\newblock {\em J. Number Theory}, 221:389--423, 2021.

\bibitem{Chen2008}
Y.-M.J. Chen.
\newblock {Average values of L-functions in characteristic two}.
\newblock {\em J. Number Theory}, 128(7):2138--2158, 2008.

\bibitem{Conrey2005}
B.~Conrey, D.W. Farmer, J.P. Keating, M.O. Rubinstein, and N.C. Snaith.
\newblock {Integral moments of L-functions}.
\newblock {\em Proc. London Math. Soc. Third Ser.}, 91(1):33--104, 2005.

\bibitem{Conrey2008}
B.~Conrey, D.W. Farmer, J.P. Keating, M.O. Rubinstein, and N.C. Snaith.
\newblock {Lower order terms in the full moment conjecture for the Riemann zeta
  function}.
\newblock {\em J. Number Theory}, 128(6):1516--1554, 2008.

\bibitem{Conrey2008a}
B.~Conrey, D.W. Farmer, and M.R. Zirnbauer.
\newblock {Autocorrelation of ratios of L-functions}.
\newblock {\em Commun. Number Theory Phys.}, 2(3):593--636, 2008.

\bibitem{ConreyGhosh1992}
B.~Conrey and A.~Ghosh.
\newblock {Mean-Value Theorems in the theory of the Riemann-zeta function}.
\newblock {\em Proceedings of the Amalfi Conference on Analytic Number Theory
  (Maiori,1989)}, pages 35--39, 1992.

\bibitem{ConreyKeatingMomentsofzetaI}
B.~Conrey and J.~Keating.
\newblock {Moments of zeta and correlations of divisor-sums: I}.
\newblock {\em Philos. Trans. R. Soc. A Math. Phys. Eng. Sci.},
  373(2040):20140313, 2015.

\bibitem{ConreyKeatingMomentsofzetaII}
B.~Conrey and J.~Keating.
\newblock {Moments of zeta and correlations of divisor-sums: II}.
\newblock {\em Adv. Theory Numbers}, eds: A.Ala(Springer New York):75--85,
  2015.

\bibitem{ConreyKeatingMomentsofzetaIII}
B.~Conrey and J.~Keating.
\newblock {Moments of zeta and correlations of divisor-sums: III}.
\newblock {\em Indag. Math.}, 26(5):736--747, 2015.

\bibitem{ConreyKeatingMomentsofzetaIV}
B.~Conrey and J.. Keating.
\newblock {Moments of zeta and correlations of divisor-sums: IV}.
\newblock {\em Res. Number Theory}, 2(1):1--24, 2016.

\bibitem{ConreyKeatingMomentsofzetaV}
B.~Conrey and J.~Keating.
\newblock {Moments of zeta and correlations of divisor-sums: V}.
\newblock {\em Proc. London Math. Soc.}, 118(4):729--752, 2019.

\bibitem{Conrey2007}
B.~Conrey and N.~C. Snaith.
\newblock {Applications of the L-Functions ratios conjectures}.
\newblock {\em Proc. London Math. Soc.}, 94(3):594--646, 2007.

\bibitem{Diaconu2019}
A.~Diaconu.
\newblock {On the third moment of L$(1/2,\chi_d)$ I: The rational function
  field case}.
\newblock {\em J. Number Theory}, 198:1--42, 2019.

\bibitem{Diaconu2004}
A.~Diaconu, D.~Goldfeld, and J.~Hoffstein.
\newblock {Multiple dirichlet series and moments of zeta and L-functions}.
\newblock {\em Compos. Math.}, 139(3):297--360, 2004.

\bibitem{Diaconu2018}
A.~Diaconu and I.~Whitehead.
\newblock {On the third moment of L(1/2,$\chi_d$) II: The number field case}.
\newblock {\em J. Eur. Math. Soc. (JEMS)}, 23(6):2051--2070, 2021.

\bibitem{Florea2017}
A.~Florea.
\newblock {Improving the Error Term in the Mean Value of
  $L\left(\frac{1}{2},\chi\right)$ in the Hyperelliptic Ensemble}.
\newblock {\em Int. Math. Res. Not. IMRN}, 20:6119--6148, 2017.

\bibitem{Florea2017a}
A.~Florea.
\newblock {The fourth moment of quadratic Dirichlet L-functions over function
  fields}.
\newblock {\em Geom. Funct. Anal.}, 27(3):541--595, 2017.

\bibitem{Florea2017c}
A.~Florea.
\newblock {The second and third moment of $L(1/2,\chi)$ in the hyperelliptic
  ensemble}.
\newblock {\em Forum Math.}, 29(4):873--892, 2017.

\bibitem{Hardy1916}
G.H. Hardy and J.E. Littlewood.
\newblock {Contributions to the theory of the Riemann zeta-function and the
  theory of the distribution of primes}.
\newblock {\em Acta Math.}, 41(1):119--196, 1916.

\bibitem{Hasse1934}
H.~Hasse.
\newblock {Thorie der relativ-zyklischen algebraischen Funktionenk\"orper
  inbesondre bei endlichen Konstantenk\"orper}.
\newblock {\em J. Reine Angew. Math.}, 172:37--54, 1934.

\bibitem{Hoffstein1992}
J.~Hoffstein and M.~Rosen.
\newblock {Average values of L-series in function fields}.
\newblock {\em J. Reine Angrew. Math}, 426:117--150, 1992.

\bibitem{Hu2010}
S.~Hu and Y.~Li.
\newblock {The genus fields of Artin-Schreier extensions}.
\newblock {\em Finite Fields Appl.}, 16(4):255--264, 2010.

\bibitem{Ingham1926}
A.E Ingham.
\newblock {Mean-Value Theorems in the theory of the Riemann-zeta function}.
\newblock {\em Proc. Lond. Math. Soc}, 27(1621):273--300, 1926.

\bibitem{Jutila1981}
M.~Jutila.
\newblock {On the Mean Value of $L(1/2,\chi)$ for Real Characters}.
\newblock {\em Analysis}, 1981.

\bibitem{Keating2000}
J.P Keating and N.C Snaith.
\newblock {Random Matrix Theory and L-Functions at s = 1/2}.
\newblock {\em Commun. Math. Phys.}, 214(1):91--100, 2000.

\bibitem{Keating2000a}
J.P. Keating and N.C. Snaith.
\newblock {Random matrix theory and $\zeta$(1/2 + it)}.
\newblock {\em Commun. Math. Phys.}, 214(1):57--89, 2000.

\bibitem{Rosen2002}
M.~Rosen.
\newblock {\em {Number Theory in Function Fields, Graduate Texts in Matematics,
  Vol. 210}}.
\newblock Springer-Verlag, New York, 2002.

\bibitem{Shen2021}
Q.~Shen.
\newblock {The fourth moment of quadratic Dirichlet L-functions}.
\newblock {\em Math. Zeitschrift}, 298:713--745, 2021.

\bibitem{Soundararajan2000NonvanishingS=1/2}
K.~Soundararajan.
\newblock {Nonvanishing of quadratic Dirichlet L-functions at $s=1/2$}.
\newblock {\em Ann. of Math.}, 152(2):447--488, 2000.

\bibitem{Stichtenoth1993}
H.~Stichtenoth.
\newblock {\em {Algebraic Function Fields and Codes}}.
\newblock Universitext, Springer-Verlang, 1993.

\end{thebibliography}
Department of Mathematics, University of Exeter, Exeter, EX4 4QF, UK\\
\textit{E-mail Address:}jm1015@exeter.ac.uk
\end{document}